\documentclass[11pt]{amsart}
\usepackage{amsthm,amsmath,amssymb,enumerate}
\usepackage{graphicx,color}
\usepackage{hyperref}
\title{Smooth approximation of the modified conical K\"ahler-Ricci flow}
\author{Ryosuke Takahashi}

\address{Mathematical Institute, Tohoku University, 6-3, Aoba, Aramaki, Aoba-ku, Sendai, 980-8578, Japan}

\email{ryosuke.takahashi.a7@tohoku.ac.jp}


\thanks{This work was supported by Grant-in-Aid for JSPS Fellows Number 16J01211.}
\makeatletter

\@addtoreset{equation}{section}
\makeatother

{
\theoremstyle{definition}

\newtheorem*{acknowledgements}{Acknowledgements}
}

{
\theoremstyle{plain}
\newtheorem{theorem}{Theorem}[section]
\newtheorem{proposition}{Proposition}[section]

}

{
\theoremstyle{remark}
\newtheorem{remark}{Remark}[section]

}
\begin{document}
%=========Abstract===================================================
\begin{abstract}
We introduce the conical K\"ahler-Ricci flow modified by a holomorphic vector field. We construct a long-time solution of the modified conical K\"ahler-Ricci flow as the limit of a sequence of smooth K\"ahler-Ricci flows.
\end{abstract}
\maketitle
%=========Index======================================================
%\tableofcontents
%=========Section 1===================================================
\section{Introduction}
Let $M$ be an $n$-dimensional Fano manifold with a K\"ahler metric $\omega_0 \in 2 \pi c_1(M)$. A K\"ahler metric $\omega \in 2 \pi c_1(M)$ is called {\it K\"ahler-Einstein} if it satisfies ${\rm Ric}(\omega)=\omega$. For a long while, it was conjectured that the existence of K\"ahler-Einstein metrics is equivalent to some algebro-geometric stability in the sense of Geometric Invariant Theory (Yau-Donaldson-Tian conjecture), which was recently solved by Chen-Donaldson-Sun \cite{CDS15} and Tian \cite{Tia15}. Their strategy was to study the existence problem of {\it smooth} K\"ahler-Einstein metrics on $M$ by deforming the cone angle, i.e., study the Gromov-Hausdorff limit of conical K\"ahler-Einstein metrics with cone angle $2 \pi \beta$ ($0 < \beta \leq 1$) along a smooth divisor $D \in |-K_M|$:
\[
{\rm Ric}(\omega)=\beta \omega+(1-\beta)[D]
\]
when $\beta$ goes to $1$, where $[D]$ is the current of integration along $D$. Although YDT conjecture has been completely settled, the existence problem of conical K\"ahler-Einstein metrics itself is also an interesting problem and studied extensively by many experts (cf. \cite{LS14}, \cite{SW16}).

Now we consider more general settings: we allow $D \in |-\lambda K_M|$ ($\lambda \in {\mathbb R}_+$) to be an ${\mathbb R}$-effective divisor with simple normal crossing support and write
\[
D=\sum_{i=1}^d \tau_i D_i
\]
where $\tau_i > 0$ and $D_i$ are smooth components. We say that a K\"ahler current $\omega \in 2 \pi c_1(M)$ is a   {\it conical K\"ahler metric} along $(1-\beta)D$ ($0 < \beta \leq 1$) if $\omega$ is smooth K\"ahler on $M \backslash D$, and asymptotically equivalent to the model conical K\"ahler metric near $D$: more precisely, near each point $p \in {\rm Supp}(D)$ where ${\rm Supp}(D)$ is cut out by the equation $\{z_1 \cdots z_r=0\}$ $(r \leq d)$ for some local holomorphic coordinates $(z^i)$, $\omega$ satisfies
\[
C^{-1} \omega_{\rm model} \leq \omega \leq C \omega_{\rm model}
\]
for some constant $C>0$, where
\[
\omega_{\rm model}:=\sqrt{-1}\sum_{i =1}^r |z^i|^{2(\beta-1)\tau_i} dz^i \wedge dz^{\bar{i}}+\sqrt{-1}\sum_{i=r+1}^n dz^i \wedge dz^{\bar{i}}
\]
is the model conical K\"ahler metric with cone angles $2 \pi (1-(1-\beta)\tau_i)$ along $\{z^i=0\}$. Let $X$ be a holomorphic vector field on $M$ whose imaginary part ${\rm Im}(X)$ generates a torus action on the line bundles ${\mathcal O}_M(D_i)$. Let $H_i$ be ${\rm Im}(X)$-invariant hermitian metrics on ${\mathcal O}_M(D_i)$ such that the curvature of the induced hermitian metric $H_D:=\otimes_{i=1}^d H_i^{\tau_i}$ is $\lambda \omega_0$. Let $s_i$ be the defining sections of ${\mathcal O}_M(D_i)$ associated to $D_i$, and set $s_D:=\otimes_{i=1}^d s_i^{\tau_i}$. We define a K\"ahler current $\omega^{\ast}$ as
\[
\omega^{\ast}:=\omega_0+k \sum_{i=1}^d \sqrt{-1}\partial \bar{\partial}|s_i|_{H_i}^{2(1-(1-\beta)\tau_i)}
\]
for sufficiently small constant $k>0$. Then $\omega^{\ast}$ is a conical K\"ahler metric along $(1-\beta)D$. According to \cite{DGSW13}, we say that a conical K\"ahler metric $\omega \in c_1(M)$ is a {\it conical K\"ahler-Ricci soliton} if it satisfies
\begin{equation} \label{cKS}
{\rm Ric}(\omega)=\gamma \omega+(1-\beta)[D]+L_X \omega
\end{equation}
in the sense of distributions on $M$, and
\[
{\rm Ric}(\omega)=\gamma \omega+L_X \omega
\]
in the classical sense on $M \backslash D$, where $\gamma=\gamma(\lambda, \beta):=1-\lambda (1-\beta) \geq 0$ and $L_X \omega$ is defined so that
\[
\int_M L_X \omega \wedge \zeta=-\int_M \omega \wedge L_X \zeta
\]
for any smooth $(n-1,n-1)$-form $\zeta$ on $M$. The notion of conical K\"ahler-Ricci solitons is a generalization of classical K\"ahler-Ricci solitons (cf. \cite{TZ00}, \cite{TZ02}) for the conical settings, and their examples in toric Fano manifolds are studied in \cite{DGSW13} and \cite{WZZ16}.

In this paper, we introduce the following {\it modified conical K\"ahler-Ricci flow} (MCKRF):
\begin{equation} \label{MCKRF}
\begin{cases}
\frac{\partial \omega}{\partial t}=-{\rm Ric}(\omega)+\gamma \omega+(1-\beta)[D]+L_X \omega\\
\omega |_{t=0}=\omega^{\ast}.
\end{cases}
\end{equation}
Then conical K\"ahler-Ricci solitons with respect to $X$ can be viewed as the stationary points of MCKRF. We say that $\omega=\omega(t)$ ($t \in [0,\infty)$) is a long-time solution of the above MCKRF if $\omega(t)$ is a conical K\"ahler metric along $(1-\beta)D$ for each $t$ which satisfies the equation \eqref{MCKRF} in the sense of distributions on $M \times [0,\infty)$ and can be simplified to the classical modified K\"ahler-Ricci flow
\[
\frac{\partial \omega}{\partial t}=-{\rm Ric}(\omega)+\gamma \omega+L_X \omega
\]
on $(M \backslash D) \times [0,\infty)$. If a long-time solution of the flow \eqref{MCKRF} converges to some K\"ahler current, it should be a conical K\"ahler-Ricci soliton with respect to $X$. Thus the flow \eqref{MCKRF} provides a new standard method for studying the equation \eqref{cKS}. In the case when $X\equiv0$, Chen-Wang \cite{CW15}\footnote{More precisely, they dealt with the ``strong'' conical K\"ahler-Ricci flow (with some H\"older continuity assumptions for potential functions).} established the short-time existence of the flow \eqref{MCKRF}. Then Liu-Zhang \cite{LZ17} and Wang \cite{Wan16} showed the long-time existence independently. On the other hand, in the general case, it seems that the flow \eqref{MCKRF} is considered only for $D=0$ (cf. \cite{TZ07}, \cite{PSSW11}).

Following the idea of \cite{LZ17} and \cite{Wan16}, we will construct a long-time solution of \eqref{MCKRF} as the limit of a  sequence of smooth K\"ahler-Ricci flows $\varphi_{\epsilon}$, where $\varphi_{\epsilon}$ ($\epsilon >0$) is a solution of the {\it modified twisted K\"ahler-Ricci flow} (MTKRF) defined in Section \ref{sec2}. Then we show the following:
\begin{theorem} \label{tck}
Assume that $|X(\log|s_D|_{H_D}^2)|<C$ on $M \backslash D$ for some constant $C>0$. Let $\omega_{\varphi_{\epsilon}}$ be a long-time solution of the modified twisted K\"ahler-Ricci flow \eqref{MTKRF}. Then, by passing to a subsequence $\{\epsilon_i\}$ satisfying $\epsilon_i \to 0$ as $i \to \infty$, the K\"ahler metric $\omega_{\varphi_{\epsilon_i}}$ converges to a solution of the modified conical K\"ahler-Ricci flow:
\[
\begin{cases}
\frac{\partial \omega_{\varphi}}{\partial t}=-{\rm Ric}(\omega_{\varphi})+\gamma \omega_{\varphi}+(1-\beta)[D]+L_X \omega_{\varphi}\\
\omega_{\varphi} |_{t=0}=\omega^{\ast}
\end{cases}
\]
as $i \to \infty$, where $\omega_{\varphi}:=\omega^{\ast}+\sqrt{-1} \partial \bar{\partial} \varphi$, and for any $t \in [0,\infty)$, the potential function $\varphi$ is H\"older continuous with respect to $\omega_0$. This covergence holds in the sense of distributions on $M \times [0,\infty)$, and in the $C_{\rm loc}^{\infty}$-topology on $(M \backslash D) \times [0,\infty)$ . In particular, there exists a long-time solution of the modified conical K\"ahler-Ricci flow.
\end{theorem}
\begin{remark}
\begin{enumerate}
\item
The assumption $|X(\log|s_D|_{H_D}^2)|<C$ is a necessary condition for the existence of a conical K\"ahler-Ricci soliton with respect to $X$. In particular, this condition implies that $X$ is tangent to ${\rm Supp}(D)$ (cf. \cite[Remark 4.2]{JLZ16}). This assumption is used only for the uniform Laplacian estimate of MTKRF (cf. Proposition \ref{ufl}).

\item
We also note that when $D$ is smooth and $\lambda \geq 1$, such a vector field $X$ automatically becomes trivial (cf. \cite[Theorem 2.1]{SW16}). This is a reason why we allow $D$ to have simple normal crossing support.
\end{enumerate}
\end{remark}
An advantage of our approach is that we do not rely on the linear theory for conical Laplacians established by Donaldson  \cite{Don12} and Chen-Wang \cite{CW15}. At the same time, we should point out that Theorem \ref{tck} provides us not only the long-time existence of solutions, but also ``the regularization method'' to study the flow. The author expects that the conical K\"ahler-Ricci flow (and its regularization) method also works for the existence problem of conical K\"ahler-Ricci solitons. The arguments in this paper run closely in parallel to those of \cite{LZ17} except some changes due to the modification $X$. Nevertheless, we will try to make the arguments reasonably self-contained for readers' convenience.

The paper is organized as follows. We first review the regularization method and reduction to the Monge-Amp\`ere flow in Section \ref{sec2}. Then we consider the uniform Laplacian estimate for MTKRF in Section \ref{sec3}. Finally, we establish the $C_{\rm loc}^{\infty}$-estimate of MTKRF and give the proof of Theorem \ref{tck} in Section \ref{sec4}.

\begin{acknowledgements}
The author would like to express his gratitude to his advisor Professor Shigetoshi Bando for useful discussions on this article. This research is supported by Grant-in-Aid for JSPS Fellows Number 16J01211.
\end{acknowledgements}
%=========Section 2===================================================
\section{Regularization and reduction to the Monge-Amp\`ere flow} \label{sec2}
Let $\epsilon >0$ be a small constant. As in \cite[Section 3.1]{GP16}, We define the function
\begin{equation}
\chi_i(\epsilon^2+u):=\frac{1}{1-(1-\beta)\tau_i} \int_0^u \frac{(\epsilon^2+r)^{1-(1-\beta)\tau_i}-\epsilon^{2 (1-(1-\beta)\tau_i)}}{r}dr
\end{equation}
for $i=1,\ldots,d$ and $u \geq 0$. Then we see that the function $\chi_i(\epsilon^2+u)$ is smooth for each $\epsilon$, and there exists uniform constants (independent of $\epsilon$) $C>0$ and $\nu>0$ such that for all $i$, we have
\begin{equation} \label{ucc}
0 \leq \chi_i(\epsilon^2+u) <C
\end{equation}
provided that $u$ belongs to a bounded interval, and
\begin{equation} \label{eso}
\omega_{\epsilon} \geq \nu \omega_0.
\end{equation}
 We also have the convergence
\[
\chi_i(\epsilon^2+|s_i|_{H_i}^2) \xrightarrow{\epsilon \to 0} |s_i|_{H_i}^{2(1-(1-\beta)\tau_i)}
\]
in the $C_{\rm loc}^{\infty}$-topology on $M \backslash D_i$. Set $\chi:=\sum_{i=1}^d \chi_i(\epsilon^2+|s_i|_{H_i}^2)$ and $\omega_{\epsilon}:=\omega_0+\sqrt{-1} \partial \bar{\partial} k\chi$. Then we have
\[
\omega_{\epsilon} \xrightarrow{\epsilon \to 0} \omega^{\ast}
\]
in the sense of distributions on $M$, and in the $C_{\rm loc}^{\infty}$-topology on $M \backslash D$.
Meanwhile, since $[D]=\lambda \omega_0+\sum_{i=1}^d\sqrt{-1}\tau_i \partial \bar{\partial} \log |s_i|_{H_i}^2$ by the Poincar\`e-Lelong formula, we observe that
\[
\eta_{\epsilon}:=\lambda \omega_0+\sum_{i=1}^d\sqrt{-1}\tau_i \partial \bar{\partial} \log (|s_i|_{H_i}^2+\epsilon^2) \xrightarrow{\epsilon \to 0} [D],
\]
again, this convergence holds in the sense of distributions on $M$, and in the $C_{\rm loc}^{\infty}$-topology on $M \backslash D$. Now We define the {\it modified twisted K\"ahler-Ricci flow} (MTKRF) with the twisted form $\eta_{\epsilon}$:
\begin{equation} \label{MTKRF}
\begin{cases}
\frac{\partial \omega_{\varphi_{\epsilon}}}{\partial t}=-{\rm Ric}(\omega_{\varphi_{\epsilon}})+\gamma\omega_{\varphi_{\epsilon}}+(1-\beta)\eta_{\epsilon}+L_X \omega_{\varphi_{\epsilon}}\\
\omega_{\varphi_{\epsilon}}|_{t=0}=\omega_{\epsilon},
\end{cases}
\end{equation}
where $\omega_{\varphi_{\epsilon}}:=\omega_{\epsilon}+\sqrt{-1}\partial \bar{\partial}\varphi_{\epsilon}$. For an ${\rm Im}(X)$-invariant K\"ahler metric $\omega \in 2 \pi c_1(M)$, we also define an ${\mathbb R}$-valued function $\theta_X(\omega)$ by
\begin{equation}
\begin{cases}
i_X \omega=\sqrt{-1} \bar{\partial} \theta_X(\omega)\\
\int_M e^{\theta_X(\omega)} \omega^n=[\omega_0]^n.
\end{cases}
\end{equation}
In particular, we set $\theta_X:=\theta_X(\omega_0)$. Then, from \cite[Proposition 1.1]{TZ02} and \cite[Corollary 5.3]{Zhu00} (or \cite[Section 2.3]{BN14}), we have the following:
\begin{proposition} \label{pfe}
Let $\phi$ be a real-valued smooth function such that ${\rm Im}(X)(\phi)=0$ and $\omega_{\phi}:=\omega_0+\sqrt{-1}\partial \bar{\partial} \phi \geq 0$. Then we have
\begin{enumerate} 
\item $\theta_X(\omega_{\phi})=\theta_X+X(\phi)$.
\item $\sup_M |X(\phi)|<C$ for some constant $C$ which depends only on $\omega_0$ and $X$.
\end{enumerate}
\end{proposition}

Since MTKRF preserves the initial K\"ahler class $[\omega_0]$, we can reduce MTKRF to the Monge-Amp\`ere flow:
\begin{equation}
\begin{cases} \label{MAF}
\frac{\partial \varphi_{\epsilon}}{\partial t}=\log \frac{\omega_{\varphi_{\epsilon}}^n}{\omega_0^n} +F_0+\gamma (k\chi+\varphi_{\epsilon})+\log(\prod_{i=1}^d (\epsilon^2+|s_i|_{H_i}^2))^{(1-\beta)\tau_i}+\theta_X(\omega_{\varphi_{\epsilon}})\\
\varphi_{\epsilon}|_{t=0}=c_{\epsilon 0}.
\end{cases}
\end{equation}
where $c_{\epsilon 0}$ is a real constant such that $c_{\epsilon 0} \xrightarrow{\epsilon \to 0} c_0$ and $F_0$ is the Ricci potential with respect to $\omega_0$:
\begin{equation}
\begin{cases}
-{\rm Ric}(\omega_0)+\omega_0=\sqrt{-1} \partial \bar{\partial}F_0\\
\int_X e^{-F_0}\omega_0^n=[\omega_0]^n.
\end{cases}
\end{equation}
We offten use the twisted Ricci potential $F_{\epsilon}$ defined by
\[
F_{\epsilon}:=F_0+\log \left(\frac{\omega_{\epsilon}^n}{\omega_0^n} \cdot \prod_{i=1}^d (\epsilon^2+|s_i|_{H_i}^2)^{(1-\beta)\tau_i} \right).
\]
\begin{remark} \label{feu}
According to \cite{CGP13}, we see that $F_{\epsilon}$ is uniformly bounded.
\end{remark}
Then the flow \eqref{MAF} can be written as
\[
\begin{cases}
\frac{\partial \varphi_{\epsilon}}{\partial t}=\log \frac{\omega_{\varphi_{\epsilon}}^n}{\omega_{\epsilon}^n} +F_{\epsilon}+\gamma (k\chi+\varphi_{\epsilon})+\theta_X(\omega_{\varphi_{\epsilon}})\\
\varphi_{\epsilon}|_{t=0}=c_{\epsilon 0}.
\end{cases}
\]
%=========Section 3===================================================
\section{$C^0$-estimate, volume ratio estimate and uniform Laplacian estimate} \label{sec3}
In this section, we establish the uniform Laplacian estimate of MTKRF. First, we show the volume ratio estimate and $C^0$-estimate:
\begin{proposition} \label{ces}
Let $\varphi_{\epsilon}$ be the solution of \eqref{MAF}. Then there exists a uniform constant $C$ (independent of $\epsilon$ and $t$) such that
\[
\sup_{M \times [0,T]} |\varphi_{\epsilon}| \leq C^{\gamma T},
\]
\[
\sup_{M \times [0,T]} |\dot{\varphi}_{\epsilon}| \leq Ce^{\gamma T}.
\]
\end{proposition}
\begin{proof}
Differentiating the equation \eqref{MAF} in $t$, we have
\[
\frac{d \dot{\varphi}_{\epsilon}}{dt}=(\Delta_{\omega_{\varphi_{\epsilon}}}+X)\dot{\varphi}_{\epsilon}+\gamma\dot{\varphi}_{\epsilon}.
\]
By the maximum principle, we have
\[
|\dot{\varphi}_{\epsilon}(t)| \leq |\dot{\varphi}(0)| e^{\gamma t},
\]
where $\dot{\varphi}(0)=F_{\epsilon}+\gamma(k\chi+c_{\epsilon 0})+\theta_X+X(k\chi)$. Thus, by \eqref{ucc}, Proposition \ref{pfe} and Remark \ref{feu}, we know that $|\dot{\varphi}(0)| \leq C$ for some uniform constant $C$. Then we have
\[
|\dot{\varphi}_{\epsilon}(t)| \leq Ce^{\gamma t}.
\]
Integrating with respect to $t$, we get
\[
|\varphi_{\epsilon}(t)| \leq Ce^{\gamma t}
\]
as desired.
\end{proof}
As in the arguments in \cite[Proposition 3.1]{LZ17} and \cite[Theorem 4.3]{JLZ16}, we can show the uniform Laplacian estimate for MTKRF:
\begin{proposition} \label{ufl}
Let $\varphi_{\epsilon}$ be a solution of \eqref{MAF}. Assume that there exists a uniform constant $C>0$ such that
\begin{enumerate}
\item $\sup_{M \times [0,T]} |\varphi_{\epsilon}|<C$,
\item $\sup_{M \times [0,T]} |\dot{\varphi}_{\epsilon}|<C$.
\end{enumerate}
Then there exists a uniform constant $A=A(\lambda,\{\tau_i\}, \beta, \omega_0, X, C)$ such that
\begin{equation} \label{ule}
A^{-1}\omega_{\epsilon} \leq \omega_{\varphi_{\epsilon}} \leq A \omega_{\epsilon}.
\end{equation}
\end{proposition}
\begin{proof}
We choose local normal coordinates $(z^i)$ with respect to $\omega_{\epsilon}$ where $\omega_{\varphi_{\epsilon}}$ is diagonal, and reduce to local computation. Then we observe that
\begin{eqnarray*}
\left( \frac{d}{dt}-\Delta_{\omega_{\varphi_{\epsilon}}} \right) \log {\rm tr}_{\omega_{\epsilon}} \omega_{\varphi_{\epsilon}}&=&\frac{1}{{\rm tr}_{\omega_{\epsilon}} \omega_{\varphi_{\epsilon}}}\left( \Delta_{\omega_{\epsilon}} \left(\dot{\varphi}_{\epsilon}-\log \frac{\omega_{\varphi_{\epsilon}}^n}{\omega_{\epsilon}^n} \right)+R_{\omega_{\epsilon}} \right)\\
&\hbox{}&-\frac{1}{{\rm tr}_{\omega_{\epsilon}}{\omega_{\varphi_{\epsilon}}}}(g_{\varphi_{\epsilon}}^{p\bar{q}} g_{\varphi_{\epsilon} j\bar{m}}R_{\omega_{\epsilon}p\bar{q}}^{\bar{m} j})\\
&\hbox{}&+\left\{ \frac{g_{\varphi_{\epsilon}}^{\delta \bar{k}} \partial_{\delta} {\rm tr}_{\omega_{\epsilon}} \omega_{\varphi_{\epsilon}} \partial_{\bar{k}} {\rm tr}_{\omega_{\epsilon}} \omega_{\varphi_{\epsilon}}}{({\rm tr}_{\omega_{\epsilon}}\omega_{\varphi_{\epsilon}})^2}-\frac{g_{\epsilon}^{\gamma \bar{s}} \varphi_{\epsilon \gamma}{}^t{}_p\varphi_{\epsilon \bar{s} t}{}^p}{{\rm tr}_{\omega_{\epsilon}} \omega_{\varphi_{\epsilon}}} \right\}.
\end{eqnarray*}
The computation in \cite[Theorem 3.9]{Tos15} implies that
\[
\frac{g_{\varphi_{\epsilon}}^{\delta \bar{k}} \partial_{\delta} {\rm tr}_{\omega_{\epsilon}} \omega_{\varphi_{\epsilon}} \partial_{\bar{k}} {\rm tr}_{\omega_{\epsilon}} \omega_{\varphi_{\epsilon}}}{({\rm tr}_{\omega_{\epsilon}}\omega_{\varphi_{\epsilon}})^2}-\frac{g_{\epsilon}^{\gamma \bar{s}} \varphi_{\epsilon \gamma}{}^t{}_p\varphi_{\epsilon \bar{s} t}{}^p}{{\rm tr}_{\omega_{\epsilon}} \omega_{\varphi_{\epsilon}}} \leq 0.
\]
Since
\[
g_{\varphi_{\epsilon}}^{p \bar{q}} g_{\varphi_{\epsilon}j \bar{m}} R_{\omega_{\epsilon}}^{\bar{m}j}{}_{p \bar{q}}=\frac{1+\varphi_{\epsilon i \bar{i}}}{1+\varphi_{\epsilon j \bar{j}}} R_{\omega_{\epsilon}}^{i\bar{i}}{}_{j\bar{j}},
\]
\[
n={\rm tr}_{\omega_{\epsilon} \omega_0}+k {\rm tr}_{\omega_{\epsilon}}(\sqrt{-1}\partial \bar{\partial}\chi) \geq k \Delta_{\omega_{\epsilon}} \chi,
\]
\[
\frac{\Delta_{\omega_{\epsilon}}\varphi_{\epsilon}}{{\rm tr}_{\omega_{\epsilon}} \omega_{\varphi_{\epsilon}}}=\frac{\sum_i \varphi_{\epsilon i \bar{i}}}{\sum_i(1+\varphi_{\epsilon i \bar{i}})} \leq 1,
\]
we have
\begin{eqnarray*}
\left( \frac{d}{dt}-\Delta_{\omega_{\varphi_{\epsilon}}} \right) \log {\rm tr}_{\omega_{\epsilon}} \omega_{\varphi_{\epsilon}} &\leq&-\frac{1}{{\rm tr}_{\omega_{\epsilon}} \omega_{\varphi_{\epsilon}}}
\sum_{i,j} \frac{1+\varphi_{\epsilon i \bar{i}}}{1+\varphi_{\epsilon j \bar{j}}} R_{\omega_{\epsilon}}^{i\bar{i}}{}_{j\bar{j}}\\
&\hbox{}& +\frac{1}{{\rm tr}_{\omega_{\epsilon}} \omega_{\varphi_{\epsilon}}} \Delta_{\omega_{\epsilon}}(F_{\epsilon}+\gamma (k\chi+\varphi_{\epsilon})+\theta_X(\omega_{\varphi_{\epsilon}}))
+R_{\omega_{\epsilon}}\\
&\leq& -\frac{1}{{\rm tr}_{\omega_{\epsilon}} \omega_{\varphi_{\epsilon}}} \sum_{i \leq j} \left( \frac{1+\varphi_{\epsilon i \bar{i}}}{1+\varphi_{\epsilon j \bar{j}}}+\frac{1+\varphi_{\epsilon j \bar{j}}}{1+\varphi_{\epsilon i \bar{i}}}-2 \right) R_{\omega_{\epsilon}}^{i\bar{i}}{}_{j\bar{j}}\\
&\hbox{}&+\frac{1}{{\rm tr}_{\omega_{\epsilon}} \omega_{\varphi_{\epsilon}}}(\Delta_{\omega_{\epsilon}}F_{\epsilon})+\frac{\gamma n}{{\rm tr}_{\omega_{\epsilon}} \omega_{\varphi_{\epsilon}}}+\gamma
+ \frac{1}{{\rm tr}_{\omega_{\epsilon}} \omega_{\varphi_{\epsilon}}} \Delta_{\omega_{\epsilon}}\theta_X(\omega_{\varphi_{\epsilon}}).
\end{eqnarray*}
Let $C_1$ be a uniform constant such that
\[
\sqrt{-1}\partial \bar{\partial}F_0 \geq -C_1 \omega_0.
\]
Then, by \eqref{eso}, we have
\[
0 \leq {\rm tr}_{\omega_{\epsilon}}(\sqrt{-1}\partial \bar{\partial}F_0+C_1\omega_0) \leq \nu^{-1} {\rm tr}_{\omega_0}(\sqrt{-1}\partial \bar{\partial}F_0+C_1\omega_0) =\nu^{-1}(C_1n+\Delta_{\omega_0} F_0).
\]
Hence we have the uniform bound of $\Delta_{\omega_{\epsilon}} F_0$:
\[
-C_1\nu^{-1} \leq - C_1{\rm tr}_{\omega_{\epsilon}} \omega_0\leq \Delta_{\omega_{\epsilon}} F_0 \leq \nu^{-1}(C_1n+\Delta_{\omega_0} F_0).
\]
Now we recall the arguments in \cite[Section 2, Section 3, Section 4]{GP16}. We set
\[
\chi_{\rho}(\epsilon^2+u)=\frac{1}{\rho} \int_0^u \frac{(\epsilon^2+r)^{\rho}-\epsilon^{2 \rho}}{r}dr
\]
and define the ``auxiliary function'' $\Psi_{\epsilon,\rho}$ by
\[
\Psi_{\epsilon,\rho}:=\widetilde{C} \sum_{i=1}^d \chi_{\rho}(\epsilon^2+|s_i|_{H_i}^2),
\]
where $\widetilde{C}>0$ and $\rho>0$ are constants. Then the function $\Psi_{\epsilon,\rho}$ is uniformly bounded.
After taking suitable uniform constants $\widetilde{C}$, $\rho$ and $C_2$, we have
\begin{eqnarray*}
-\sum_{i \geq j} \left( \frac{1+\varphi_{\epsilon i \bar{i}}}{1+\varphi_{\epsilon j \bar{j}}}+\frac{1+\varphi_{\epsilon j \bar{j}}}{1+\varphi_{\epsilon i \bar{i}}}-2 \right) R_{\omega_{\epsilon}}^{i\bar{i}}{}_{j\bar{j}}-{\rm tr}_{\omega_{\epsilon}}\omega_{\varphi_{\epsilon}} \Delta_{\omega_{\varphi_{\epsilon}}} \Psi_{\epsilon, \rho}+\Delta_{\omega{\epsilon}}F_\epsilon \\
\leq C_2 \sum_{i \leq j} \left( \frac{1+\varphi_{\epsilon i \bar{i}}}{1+\varphi_{\epsilon j \bar{j}}}+\frac{1+\varphi_{\epsilon j \bar{j}}}{1+\varphi_{\epsilon i \bar{i}}} \right)+C_2 {\rm tr}_{\omega_{\varphi_{\epsilon}}} \omega_{\epsilon} \cdot {\rm tr}_{\omega_{\epsilon}}\omega_{\varphi_{\epsilon}}+\Delta_{\omega_{\epsilon}} F_0+ C_2.
\end{eqnarray*}
Combining with the Cauchy-Shwartz inequality $n \leq {\rm tr}_{\omega_{\varphi_{\epsilon}}} \omega_{\epsilon} \cdot {\rm tr}_{\omega_{\epsilon}} \omega_{\varphi_{\epsilon}}$, we get
\begin{eqnarray*}
\left( \frac{d}{dt}-\Delta_{\omega_{\varphi_{\epsilon}}} \right)(\log {\rm tr}_{\omega_{\epsilon}} \omega_{\varphi_{\epsilon}}+\Psi_{\epsilon,\rho})&\leq& \frac{C_2}{{\rm tr}_{\omega_{\epsilon}} \omega_{\varphi_{\epsilon}}} \sum_{i \leq j} \left( \frac{1+\varphi_{\epsilon i \bar{i}}}{1+\varphi_{\epsilon j \bar{j}}}+\frac{1+\varphi_{\epsilon j \bar{j}}}{1+\varphi_{\epsilon i \bar{i}}} \right)\\
&\hbox{}& +\frac{C_3}{{\rm tr}_{\omega_{\epsilon}} \omega_{\varphi_{\epsilon}}}+C_2{\rm tr}_{\omega_{\varphi_{\epsilon}}}\omega_{\epsilon}+\frac{1}{{\rm tr}_{\omega_{\epsilon}} \omega_{\varphi_{\epsilon}}} \Delta_{\omega_{\epsilon}}\theta_X(\omega_{\varphi_{\epsilon}})+C_4\\
&\leq&  \frac{C_2}{{\rm tr}_{\omega_{\epsilon}}\omega_{\varphi_{\epsilon}}} \left\{ \left( \sum_i \frac{1}{1+\varphi_{\epsilon i \bar{i}}} \right) \left( \sum_j(1+\varphi_{\epsilon j \bar{j}}) \right) +n \right\} \\
&\hbox{}& +\frac{C_3}{{\rm tr}_{\omega_{\epsilon}} \omega_{\varphi_{\epsilon}}}+C_2{\rm tr}_{\omega_{\varphi_{\epsilon}}}\omega_{\epsilon}+\frac{1}{{\rm tr}_{\omega_{\epsilon}} \omega_{\varphi_{\epsilon}}} \Delta_{\omega_{\epsilon}}\theta_X(\omega_{\varphi_{\epsilon}})+C_4\\
&\leq& C_5{\rm tr}_{\omega_{\varphi_{\epsilon}}} \omega_{\epsilon}+\frac{1}{{\rm tr}_{\omega_{\epsilon}} \omega_{\varphi_{\epsilon}}} \Delta_{\omega_{\epsilon}}\theta_X(\omega_{\varphi_{\epsilon}})+C_4.
\end{eqnarray*}
Since $\max_{i=1,\ldots,n} \{\sup_M X^i{}_{,i},0\} \leq C_6$ is uniformly bounded\footnote{We need the assumption $|X(\log|s_D|_{H_D}^2)|<C$ to get this uniform bound.} (cf. \cite[Lemma A.2]{JLZ16}), we get
\begin{eqnarray*}
\Delta_{\omega_{\epsilon}}\theta_X(\omega_{\varphi_{\epsilon}})&=& \sum_i \theta_X(\omega_{\varphi_{\epsilon}})_{i \bar{i}}\\
&=& \sum_i(X^j g_{\varphi_{\epsilon}j \bar{i}})_i \\
&=& \sum_i(X^j{}_{,i}g_{\varphi_{\epsilon}j \bar{i}}+X^j g_{\varphi_{\epsilon}j \bar{i},i})\\
&=& \sum_i(X^j{}_{,i}g_{\varphi_{\epsilon}j \bar{i}}+X^j g_{\varphi_{\epsilon}i \bar{i},j})\\
&=& \sum_i X^i{}_{,i}(1+\varphi_{\epsilon i \bar{i}})+\sum_i X^j \varphi_{{\epsilon}i \bar{i} j}\\
&\leq& C_6 {\rm tr}_{\omega_{\epsilon}} \omega_{\varphi_{\epsilon}}+\sum_i X^j \varphi_{{\epsilon}i \bar{i} j}.
\end{eqnarray*}
On the other hand, from the assumption (2), we know that
\[
\left(\frac{d}{dt}-\Delta_{\omega_{\varphi_{\epsilon}}} \right)\varphi_{\epsilon}=\dot{\varphi}_{\epsilon}-{\rm tr}_{\omega_{\varphi_{\epsilon}}}(\omega_{\varphi_{\epsilon}}-\omega_{\epsilon})
=\dot{\varphi}-n+{\rm tr}_{\omega_{\varphi_{\epsilon}} \omega_{\epsilon}}
\geq {\rm tr}_{\omega_{\varphi_{\epsilon}}} \omega_{\epsilon}-(C+n).
\]
Thus, if we set $B:=C_5+1$, we have
\[
\left( \frac{d}{dt}-\Delta_{\omega_{\varphi_{\epsilon}}} \right)(\log {\rm tr}_{\omega_{\epsilon}} \omega_{\varphi_{\epsilon}}+\Psi_{\epsilon,\rho}-B \varphi_{\epsilon}) \leq - {\rm tr}_{\omega_{\varphi_{\epsilon}}} \omega_{\epsilon}+\frac{1}{{\rm tr}_{\omega_{\epsilon}}\omega_{\varphi_{\epsilon}}} \sum_i X^j \varphi_{\epsilon j \bar{j} i}+C_7.
\]
We assume that the function $\log {\rm tr}_{\omega_{\epsilon}} \omega_{\varphi_{\epsilon}}+\Psi_{\epsilon,\rho}-B \varphi_{\epsilon}$ takes its maximum at $(x_0,t_0) \in M \times [0,T]$. If $t_0=0$, we have $\log {\rm tr}_{\omega_{\epsilon}} \omega_{\varphi_{\epsilon}}+\Psi_{\epsilon,\rho}-B \varphi_{\epsilon} = \log n+\Psi_{\epsilon,\rho}-B c_{\epsilon 0}$, which is uniformly bounded since $\Psi_{\epsilon,\rho}$ and $c_{\epsilon 0}$ is. Now we assume that $t_0>0$. Then,  by the maximum principle, we have
\[
0 \leq - {\rm tr}_{\omega_{\varphi_{\epsilon}}} \omega_{\epsilon}+\frac{1}{{\rm tr}_{\omega_{\epsilon}}\omega_{\varphi_{\epsilon}}} \sum_i X^j \varphi_{\epsilon j \bar{j} i}+C_7
\]
at $(x_0,t_0)$. On the other hand, differentiating the function $\log {\rm tr}_{\omega_{\epsilon}} \omega_{\varphi_{\epsilon}}+\Psi_{\epsilon,\rho}-B \varphi_{\epsilon}$ in $z^j$ implies
\[
\frac{\partial}{\partial z^j}(\log {\rm tr}_{\omega_{\epsilon}} \omega_{\varphi_{\epsilon}}+\Psi_{\epsilon,\rho}-B \varphi_{\epsilon})=\frac{1}{{\rm tr}_{\omega_{\epsilon}}\omega_{\varphi_{\epsilon}}} \varphi_{\epsilon i \bar{i} j}+\Psi_{\epsilon, \rho,j}-B\varphi_{\epsilon j}.
\]
Hence, at $(x_0,t_0)$, we have
\begin{eqnarray*}
\frac{1}{{\rm tr}_{\omega_{\epsilon}}\omega_{\varphi_{\epsilon}}} \sum_i X^j \varphi_{\epsilon j \bar{j} i}=X(B\varphi_{\epsilon}-\Psi_{\epsilon,\rho}).
\end{eqnarray*}
According to \cite[Section 4]{GP16}, we find that there exists a small uniform constant $k'>0$ such that $\omega_0+k' \sqrt{-1}\partial \bar{\partial}\Psi_{\epsilon,\rho} \geq 0$. Thus, combining with Proposition \ref{pfe} implies
\[
|X(\varphi_{\epsilon})| \leq |X(k\chi+\varphi_{\epsilon})|+|X(k\chi)| \leq C_8,
\]
\[
|X(\Psi_{\epsilon,\rho})| \leq C_9.
\]
Thus we have
\[
{\rm tr}_{\omega_{\varphi_{\epsilon}}} \omega_{\epsilon} \leq C_{10}
\]
at $(x_0,t_0)$. Then we observe that
\begin{eqnarray*}
{\rm tr}_{\omega_{\epsilon}} \omega_{\varphi_{\epsilon}}(x_0,t_0) &\leq& \frac{1}{(n-1)!}({\rm tr}_{\omega_{\varphi_{\epsilon}}} \omega_{\epsilon})^{n-1}(x_0,t_0) \cdot \frac{\omega_{\varphi_{\epsilon}^n}}{\omega_{\epsilon}^n}(x_0,t_0)\\
&\leq&\frac{C_{10}^{n-1}}{(n-1)!} \exp(\dot{\varphi}_{\epsilon}-F_{\epsilon}-\gamma (k\chi+\varphi_{\epsilon})-\theta_X-X(k \chi+\varphi_{\epsilon}))(x_0,t_0)\\
&\leq& C_{11}.
\end{eqnarray*}
Since $F_{\epsilon}$ and $\Psi_{\epsilon,\rho}$ are uniformly bounded, we find that
\[
{\rm tr}_{\omega_{\epsilon}} \omega_{\varphi_{\epsilon}} \leq C_{12}
\]
on $M$. Hence the flow equation \eqref{MAF} and the uniform bound of $\varphi_{\epsilon}$, $\dot{\varphi}_{\epsilon}$, $F_{\epsilon}$, $X(k\chi+\varphi_{\epsilon})$ give the desired inequality \eqref{ule} for some uniform constant $A$.
\end{proof}
%=========Section 4===================================================
\section{$C_{\rm loc}^{\infty}$-estimate and completion of the proof of Theorem \ref{tck}} \label{sec4}
In this section, we establish the $C_{\rm loc}^{\infty}$-estimate of MTKRF. Let
\[
\phi_{\epsilon}:=\varphi_{\epsilon}+k\chi.
\]
Then we have
\[
\omega_{\phi_{\epsilon}}:=\omega_0+\sqrt{-1}\partial \bar{\partial}\phi=\omega_{\varphi_{\epsilon}}.
\]
In order to simplify the notation, we drop the explicit dependence of $\epsilon$ and write $\phi$, $\eta$, etc. Then the equation of MTKRF can be written as
\begin{equation} \label{sMAf}
\frac{\partial \omega_{\phi}}{\partial t}=-{\rm Ric}(\omega_{\phi})+\gamma \omega_{\phi}+\widetilde{\eta}+L_X\omega_{\phi},
\end{equation}
where $\widetilde{\eta}:=(1-\beta) \eta \in (1-\gamma) c_1(M)$, or equivalently,
\begin{equation} \label{sMAc}
\frac{d g_{\phi k \bar{l}}}{dt}=-R_{\phi k \bar{l}}+\gamma g_{\phi k \bar{l}}+\widetilde{\eta}_{k \bar{l}}+\nabla_{\phi k} X_{\bar{l}}.
\end{equation}
Then we can reduce the above equation to the Monge-Amp\`ere flow:
\begin{equation} \label{MAf}
\frac{\partial \phi}{\partial t}=\log\frac{\omega_{\phi}^n}{\omega_0^n}+\gamma \phi+F+\theta_X(\omega_{\phi}),
\end{equation}
where $F$ is a twisted Ricci potential $\sqrt{-1}\partial \bar{\partial}F=-{\rm Ric}(\omega_0)+\gamma \omega_0+\widetilde{\eta}$. Let $\nabla_{\phi}$ (resp. $\nabla_0$) be the covariant derivative with respect to $\omega_{\phi}$ (resp. $\omega_0$). We set
\[
S:=|\nabla_0 g_{\phi}|_{\omega_{\phi}}^2=g_{\phi}^{i\bar{j}}g_{\phi}^{k\bar{l}}g_{\phi}^{p\bar{q}}\nabla_{0i} g_{\phi k \bar{q}} \nabla_{0\bar{j}} g_{\phi p \bar{l}}.
\]
If we put
\[
h^i{}_k:=g_0^{i\bar{j}}g_{\phi k\bar{j}},
\]
\[
U_{il}^k:=(\nabla_{\phi i}h \cdot h^{-1})^k{}_l,
\]
then we have
\begin{equation} \label{UgC}
U_{il}^k=\Gamma_{\phi i l}^k-\Gamma_{0 i l}^k,
\end{equation}
\[
S=|U|_{\omega_{\phi}}^2,
\]
where $\Gamma_{\phi i l}^k$ (resp. $\Gamma_{0 i l}^k$) is the Christoffel symbol of $\omega_{\phi}$ (resp. $\omega_0$). The following proposition is an $X$-analogue of \cite[Proposition 3.3]{LZ17}.
\begin{proposition} \label{toe}
Let $p \in M$ and $\phi$ be a solution of the Monge-Amp\`ere flow \eqref{MAf}. We assume that there exists a constant $N>0$ such that
\begin{equation} \label{ulc}
N^{-1}\omega_0 \leq \omega_{\phi} \leq N \omega_0
\end{equation}
on $B_r(p) \times [0,T]$, where $B_r(p)$ is a geodesic ball of radius $r>0$ centered at $p$ with respect to $\omega_0$. Then there exists constants
\[
C'=C'(N, \gamma, \omega_0, X, \| \phi (\cdot,0) \|_{C^3(B_r(p))}, \| \widetilde{\eta} \|_{C^1(B_r(p))})
\]
and
\[
C''=C''(N, \gamma, \omega_0, X, \| \phi (\cdot,0) \|_{C^4(B_r(p))}, \| \widetilde{\eta} \|_{C^2(B_r(p))})
\]
such that
\[
S \leq C',
\]
\[
|{\rm Rm}_{\phi}|_{\omega_{\phi}}^2 \leq C''
\]
on $B_{r/2}(p) \times [0,T]$.
Moreover, for any $k \geq 0$ and $0<\alpha<1$, there exists constants
\[
C_k^i=C_k^i(N, \gamma, \omega_0, X, \|\phi(\cdot,0)\|_{C^{k+4}(B_r(p))}, \|\phi \|_{C^0(B_r(p) \times [0,T])}, \|\widetilde{\eta}\|_{C^{k+2}(B_r(p))}, \|F\|_{C^0(B_r(p))}) \;\; (i=1,2,3)
\]
such that
\[
|D^k {\rm Rm}_{\phi}|_{\omega_{\phi}}^2 \leq C_k^1,
\]
\[
\| \dot{\phi} \|_{C^{k+1, \alpha}} \leq C_k^2,
\]
\[
\| \phi \|_{C^{k+3, \alpha}} \leq C_k^3
\]
on $B_{r/2}(p) \times [0,T]$.
\end{proposition}
\begin{proof}
We first establish the local version of Calabi's $C^3$-esitimate. A direct computation shows that
\begin{eqnarray*}
\left(\frac{d}{dt}-\Delta_{\omega_{\phi}} \right)S&=&g_{\phi}^{m\bar{\gamma}} g_{\phi \bar{\mu} \beta} g_{\phi}^{l\bar{\alpha}}((g_{\phi}^{\beta \bar{s}} \nabla_{\phi m} \widetilde{\eta}_{\bar{s}l}-\nabla_{\phi}^{\bar{q}}R_0{}^{\beta}{}_{l\bar{q}m})U_{\bar{\gamma}\bar{\alpha}}^{\bar{\mu}}+U_{ml}^{\beta}(g_{\phi}^{\bar{\mu}s} \nabla_{\phi \bar{\gamma}} \widetilde{\eta}_{s\bar{\alpha}}-\nabla_{\phi}^q R_0{}^{\bar{\mu}}{}_{\bar{\alpha}q\bar{\gamma}}))\\
&\hbox{}& -U_{ml}^{\beta} U_{\bar{\gamma}\bar{\alpha}}^{\bar{\mu}}(\widetilde{\eta}_{p \bar{q}}g_{\phi}^{p \bar{\gamma}}g_{\phi}^{m\bar{q}}g_{\phi \bar{\mu} \beta}g_{\phi}^{l \bar{\alpha}}-g_{\phi}^{m\bar{\gamma}}\widetilde{\eta}_{\bar{\mu}\beta}g_{\phi}^{l\bar{\alpha}}+g_{\phi}^{m\bar{\gamma}}g_{\phi \bar{\mu} \beta}g_{\phi}^{p\bar{\alpha}}g_{\phi}^{l\bar{q}}\widetilde{\eta}_{p \bar{q}}) \\
&\hbox{}& -\gamma S-|\nabla_{\phi}U|_{\omega_{\phi}}^2-|\overline{\nabla}_{\phi}U|_{\omega_{\phi}}^2\\
&\hbox{}& +\underbrace{g_{\phi}^{m\bar{\gamma}} g_{\phi \bar{\mu} \beta} g_{\phi}^{l\bar{\alpha}} \cdot \nabla_{\phi m} \nabla_{\phi l}X^{\beta} \cdot U_{\bar{\gamma}\bar{\alpha}}^{\bar{\mu}}}_\text{($X$;I)}+
\underbrace{g_{\phi}^{m\bar{\gamma}} g_{\phi \bar{\mu} \beta} g_{\phi}^{l \bar{\alpha}} \cdot \nabla_{\phi \bar{\gamma}} \nabla_{\phi \bar{\alpha}} X^{\bar{\mu}} \cdot U_{ml}^{\beta}}_\text{($X$;II)}\\
&\hbox{}&-\underbrace{g_{\phi \bar{\mu} \beta} g_{\phi}^{l \bar{\alpha}}\cdot \nabla_{\phi}^{\bar{\gamma}} X^m \cdot U_{ml}^{\beta} U_{\bar{\gamma}\bar{\alpha}}^{\bar{\mu}}}_\text{($X$;III)}
+\underbrace{g_{\phi}^{m\bar{\gamma}}g_{\phi}^{l\bar{\alpha}} \cdot \nabla_{\phi \beta} X_{\bar{\mu}} \cdot U_{ml}^{\beta} U_{\bar{\gamma}\bar{\alpha}}^{\bar{\mu}}}_\text{($X$;IV)}\\
&\hbox{}& -\underbrace{g_{\phi}^{m \bar{\gamma}} g_{\phi \bar{\mu} \beta} \cdot \nabla_{\phi}^{\bar{\alpha}} X^l \cdot U_{ml}^{\beta} U_{\bar{\gamma}\bar{\alpha}}^{\bar{\mu}}}_\text{($X$;V)},
\end{eqnarray*}
where ($X$;I)-($X$;V) are additional terms arising from the holomorphic vector field $X$.
Since
\begin{equation} \label{foe}
\nabla_{\phi m} \widetilde{\eta}_{l \bar{q}}=\nabla_{0m} \widetilde{\eta}_{l \bar{q}}-U_{ml}^s \widetilde{\eta}_{s \bar{q}},
\end{equation}
\begin{equation}
\nabla_{\phi p}R_0{}^{\beta}{}_{l \bar{q} m}=\nabla_{0p}R_0{}^{\beta}{}_{l\bar{q}m}+U_{ps}^{\beta}R_0{}^s{}_{l \bar{q} m}-U_{pl}^s R_0{}^{\beta}{}_{s \bar{q} m}-U_{pm}^s R_0{}^{\beta}{}_{l \bar{q} s}.
\end{equation}
we have
\begin{eqnarray*}
g_{\phi}^{m\bar{\gamma}} g_{\phi \bar{\mu} \beta} g_{\phi}^{l\bar{\alpha}}((g_{\phi}^{\beta \bar{s}} \nabla_{\phi m} \widetilde{\eta}_{\bar{s}l}-\nabla_{\phi}^{\bar{q}}R_0{}^{\beta}{}_{l\bar{q}m})U_{\bar{\gamma}\bar{\alpha}}^{\bar{\mu}}+U_{ml}^{\beta}(g_{\phi}^{\bar{\mu}s} \nabla_{\phi \bar{\gamma}} \widetilde{\eta}_{s\bar{\alpha}}-\nabla_{\phi}^q R_0{}^{\bar{\mu}}{}_{\bar{\alpha}q\bar{\gamma}}))\\
- U_{ml}^{\beta} U_{\bar{\gamma}\bar{\alpha}}^{\bar{\mu}}(\widetilde{\eta}_{p \bar{q}}g_{\phi}^{p \bar{\gamma}}g_{\phi}^{m\bar{q}}g_{\phi \bar{\mu} \beta}g_{\phi}^{l \bar{\alpha}}-g_{\phi}^{m\bar{\gamma}}\widetilde{\eta}_{\bar{\mu}\beta}g_{\phi}^{l\bar{\alpha}}+g_{\phi}^{m\bar{\gamma}}g_{\phi \bar{\mu} \beta}g_{\phi}^{p\bar{\alpha}}g_{\phi}^{l\bar{q}}\widetilde{\eta}_{p \bar{q}})-\gamma S \leq C_1(S+1),
\end{eqnarray*}
where the constant $C_1$ depends only on $N$, $\gamma$, $\omega_0$ and $\| \widetilde{\eta} \|_{C^1(B_r(p))}$. On the other hand, since
\[
\nabla_{\phi l}X^{\beta}=\nabla_{0l}X^{\beta}+X^k U_{lk}^{\beta},
\]
\[
\nabla_{\phi m} \nabla_{\phi l} X^{\beta}=\nabla_{0m} \nabla_{0l} X^{\beta}-\nabla_{0p}X^{\beta} \cdot U_{ml}^p+\nabla_{0l}X^p \cdot U_{pm}^{\beta}+\nabla_{\phi m}X^k \cdot U_{lk}^{\beta}+X^k \nabla_{\phi m}U_{lk}^{\beta},
\]
in the same way as \cite[Section 6]{PSSW11}, we observe that
\[
|\text{($X$;III)}|+|\text{($X$;IV)}|+|\text{($X$;V)}| \leq C_2 S|\nabla_{\phi}X|_{\omega_{\phi}},
\]
\begin{eqnarray*}
|\text{($X$;I)}|+|\text{($X$;II)}| &\leq& C_3(S+1)+S|\nabla_{\phi}X|_{\omega_{\phi}}+|X|_{\omega_{\phi}}|U|_{\omega_{\phi}}|\nabla_{\phi}U|_{\omega_{\phi}} \\
&\leq& C_3(S+1)+S|\nabla_{\phi}X|_{\omega_{\phi}}+\frac{1}{2} |\nabla_{\phi}U|_{\omega_{\phi}}^2+\frac{1}{2}|X|_{\omega_{\phi}}^2 |U|_{\omega_{\phi}}^2\\
&\leq& C_4(S+1)+\frac{1}{2}|\nabla_{\phi}U|_{\omega_{\phi}}^2+S|\nabla_{\phi}X|_{\omega_{\phi}},
\end{eqnarray*}
where $C_4$ depends only on $X$, $\omega_0$ and $N$. Thus we have
\begin{equation} \label{eoS}
\left( \frac{d}{dt}-\Delta_{\omega_{\phi}} \right)S \leq -\frac{1}{2}|\nabla_{\phi}U|_{\omega_{\phi}}^2-|\overline{\nabla}_{\phi}U|_{\omega_{\phi}}^2+(C_2+1) S|\nabla_{\phi}X|_{\omega_{\phi}}+(C_1+C_4)(S+1).
\end{equation}
On the other hand, the evolution equation of $|X|_{\omega_{\phi}}^2$ can be estimated as
\begin{eqnarray} \label{evf}
\left( \frac{d}{dt}-\Delta_{\omega_{\phi}} \right)|X|_{\omega_{\phi}}^2&=&\gamma |X|_{\omega_{\phi}}^2-|\nabla_{\phi}X|_{\omega_{\phi}}^2+(\widetilde{\eta}_{i\bar{j}}+\nabla_{\phi i}X_{\bar{j}})X^iX^{\bar{j}} \nonumber \\
&\leq&-\frac{1}{2}|\nabla_{\phi}X|_{\omega_{\phi}}^2+C_5.
\end{eqnarray}
Now we work in local normal coordinates $(z^i)$ with respect to $\omega_0$ where $\omega_{\phi}$ is diagonal. Since
\[
0 \leq {\rm tr}h \leq nN,
\]
\[
g_0^{j \bar{s}}g_{\phi}^{p \bar{q}} g_{\phi}^{m \bar{k}} \phi_{j \bar{k} p} \phi_{\bar{s} m \bar{q}} \geq \frac{1}{N} S,
\]
\[
|g_0^{i\bar{j}}\nabla_{\phi i}X_{\bar{j}}| \leq {\rm tr}h \cdot |{\rm tr} \nabla_{\phi} X| \leq C_6(S^{1/2}+1) \leq \frac{1}{N+1}S+C_7,
\]
we observe that
\begin{eqnarray} \label{eqh}
(\frac{d}{dt}-\Delta_{\omega_{\phi}}) {\rm tr}h&=&\gamma {\rm tr}h+g_0^{i \bar{j}}(\widetilde{\eta}_{i \bar{j}}+\nabla_{\phi i}X_{\bar j})-g_{\phi}^{p \bar{q}} g_0^{\beta \bar{\gamma}} g_{\phi \alpha \bar{\gamma}} R_0{}^{\alpha}{}_{\beta \bar{q} p}-g_0^{j \bar{s}}g_{\phi}^{p \bar{q}} g_{\phi}^{m \bar{k}} \phi_{j \bar{k} p} \phi_{\bar{s} m \bar{q}} \nonumber \\
&\leq& C_8-\frac{1}{N(N+1)}S.
\end{eqnarray}

Let $r>r_1>r/2$ and $\kappa$ be a nonnegative smooth cut-off function that is identically equal to $1$ on $\overline{B_{r_1}(p)}$ and vanishes on the outside of $B_r(p)$. Furthermore, we assume that
\[
|\partial \kappa|_{\omega_0}, \;\; |\sqrt{-1}\partial \bar{\partial} \kappa |_{\omega_0} \leq C_9.
\]
We consider the function
\[
W:=\kappa^2 \frac{S}{K-|X|_{\omega_{\phi}}^2}+A {\rm tr}h,
\]
where $K$ is a uniform constant such that $\frac{256}{257} K \leq K-|X|_{\omega_{\phi}}^2 \leq K$ and $A$ is a uniform constant determined later. A direct computation shows that
\begin{eqnarray*}
\left( \frac{d}{dt}-\Delta_{\omega_{\phi}} \right) W &=& (-\Delta_{\omega_{\phi}} \kappa^2) \frac{S}{K-|X|_{\omega_{\phi}}^2}-4{\rm Re}\left( \frac{\kappa \nabla_{\phi} \kappa}{K-|X|_{\omega_{\phi}}^2}, \nabla_{\phi} S \right)_{\omega_{\phi}} \\
&\hbox{}& -4 {\rm Re} \left( \kappa \nabla_{\phi} \kappa, \frac{S \cdot \nabla_{\phi} |X|_{\omega_{\phi}}^2}{(K-|X|_{\omega_{\phi}}^2)^2} \right)_{\omega_{\phi}} + \frac{\kappa^2}{K-|X|_{\omega_{\phi}}^2} \left( \frac{d}{dt}-\Delta_{\omega_{\phi}} \right) S\\
&\hbox{}&+\frac{\kappa^2 S}{(K-|X|_{\omega_{\phi}}^2)^2} \left( \frac{d}{dt}-\Delta_{\omega_{\phi}} \right) |X|_{\omega_{\phi}}^2 -\frac{2 \kappa^2 S (\nabla_{\phi}|X|_{\omega_{\phi}}^2)^2}{(K-|X|_{\omega_{\phi}}^2)^3}\\
&\hbox{}& -\frac{2 \kappa^2 {\rm Re}(\nabla_{\phi}|X|_{\omega_{\phi}}^2, \nabla_{\phi} S)_{\omega_{\phi}}}{(K-|X|_{\omega_{\phi}}^2)^2}+A \left( \frac{d}{dt}-\Delta_{\omega_{\phi}} \right) {\rm tr}h.
\end{eqnarray*}
Using \eqref{eoS}, \eqref{evf} and the facts
\begin{equation}
|\nabla_{\phi}|X|_{\omega_{\phi}}^2|_{\omega_{\phi}} \leq |X|_{\omega_{\phi}} |\nabla_{\phi}X|_{\omega_{\phi}},
\end{equation}
\begin{equation} \label{aes}
|\nabla_{\phi}S|_{\omega_{\phi}}^2 \leq 2S(|\nabla_{\phi}U|_{\omega_{\phi}}^2+|\overline{\nabla}_{\phi}U|_{\omega_{\phi}}^2),
\end{equation}
we observe that
\[
\left|(-\Delta_{\omega_{\phi}} \kappa^2) \frac{S}{K-|X|_{\omega_{\phi}}^2} \right| \leq C_{10}S,
\]
\begin{eqnarray*}
\left| 4{\rm Re}\left( \frac{\kappa \nabla_{\phi} \kappa}{K-|X|_{\omega_{\phi}}^2}, \nabla_{\phi} S \right)_{\omega_{\phi}} \right| &\leq& \frac{4\sqrt{2}}{K-|X|_{\omega_{\phi}}^2} \kappa |\nabla_{\phi} \kappa|_{\omega_{\phi}} S^{1/2} (|\nabla_{\phi}U|_{\omega_{\phi}}^2+|\overline{\nabla}_{\phi}U|_{\omega_{\phi}}^2)^{1/2}\\
&\leq& C_{11}S+\frac{\kappa^2}{4(K-|X|_{\omega_{\phi}}^2)}(|\nabla_{\phi}U|_{\omega_{\phi}}^2+|\overline{\nabla}_{\phi}U|_{\omega_{\phi}}^2),
\end{eqnarray*}
\[
\left| 4 {\rm Re} \left( \kappa \nabla_{\phi} \kappa, \frac{S \cdot \nabla_{\phi} |X|_{\omega_{\phi}}^2}{(K-|X|_{\omega_{\phi}}^2)^2} \right)_{\omega_{\phi}} \right| \leq C_{12}S+\frac{\kappa^2S|\nabla_{\phi}X|_{\omega_{\phi}}^2}{4(K-|X|_{\omega_{\phi}}^2)^2},
\]
\begin{eqnarray*}
\frac{\kappa^2}{K-|X|_{\omega_{\phi}}^2} \left( \frac{d}{dt}-\Delta_{\omega_{\phi}} \right)S &\leq& -\frac{\kappa^2}{2(K-|X|_{\omega_{\phi}}^2)}(|\nabla_{\phi}U|_{\omega_{\phi}}^2+|\overline{\nabla}_{\phi}U|_{\omega_{\phi}}^2)+\frac{(C_2+1)\kappa^2S|\nabla_{\phi}X|_{\omega_{\phi}}}{K-|X|_{\omega_{\phi}}^2}\\
&\hbox{}& +\frac{\kappa^2(C_1+C_4)}{K-|X|_{\omega_{\phi}}^2}(S+1)\\
&\leq& -\frac{\kappa^2}{2(K-|X|_{\omega_{\phi}}^2)}(|\nabla_{\phi}U|_{\omega_{\phi}}^2+|\overline{\nabla}_{\phi}U|_{\omega_{\phi}}^2)+\frac{\kappa^2S|\nabla_{\phi}X|_{\omega_{\phi}}^2}{8(K-|X|_{\omega_{\phi}}^2)^2}\\
&\hbox{}& +C_{13}(S+1),
\end{eqnarray*}
\[
\frac{\kappa^2 S}{(K-|X|_{\omega_{\phi}}^2)^2} \left( \frac{d}{dt}-\Delta_{\omega_{\phi}} \right) |X|_{\omega_{\phi}}^2 \leq -\frac{\kappa^2 S |\nabla_{\phi}X|_{\omega_{\phi}}^2}{2(K-|X|_{\omega_{\phi}}^2)^2}+C_{14}S,
\]
\begin{eqnarray*}
\left| \frac{2 \kappa^2 {\rm Re}(\nabla_{\phi}|X|_{\omega_{\phi}}^2, \nabla_{\phi} S)_{\omega_{\phi}}}{(K-|X|_{\omega_{\phi}}^2)^2} \right| &\leq& \frac{2 \sqrt{2} \kappa^2}{(K-|X|_{\omega_{\phi}}^2)^2}|X|_{\omega_{\phi}}|\nabla_{\phi}X|_{\omega_{\phi}} S^{1/2}(|\nabla_{\phi}U|_{\omega_{\phi}}^2+|\overline{\nabla}_{\phi}U|_{\omega_{\phi}}^2)^{1/2}\\
&\leq& \frac{\kappa^2S|\nabla_{\phi}X|_{\omega_{\phi}}^2}{16(K-|X|_{\omega_{\phi}}^2)^2}+\frac{32\kappa^2|X|_{\omega_{\phi}}^2}{(K-|X|_{\omega_{\phi}}^2)^2}(|\nabla_{\phi}U|_{\omega_{\phi}}^2+|\overline{\nabla}_{\phi}U|_{\omega_{\phi}}^2)\\
&\leq& \frac{\kappa^2S|\nabla_{\phi}X|_{\omega_{\phi}}^2}{16(K-|X|_{\omega_{\phi}}^2)^2}+\frac{\kappa^2}{8(K-|X|_{\omega_{\phi}}^2)}(|\nabla_{\phi}U|_{\omega_{\phi}}^2+|\overline{\nabla}_{\phi}U|_{\omega_{\phi}}^2)\\
&\hbox{}& \text{(because $\frac{256}{257}K < K-|X|_{\omega_{\phi}}^2<K$)}.
\end{eqnarray*}
Hence, combining with \eqref{eqh}, we get
\[
\left( \frac{d}{dt}-\Delta_{\omega_{\phi}} \right) W \leq \left( C_{10}+C_{11}+C_{13}+C_{14}-\frac{A}{N(N+1)} \right) S+C_{13}.
\]
Let $(x_0,t_0)$ be the maximum point of $W$ on $\overline{B_r(p)} \times [0,T]$. If $t_0=0$, then $S$ is bounded by the initial data $\|\phi(\cdot,0)\|_{C^3(B_r(p))}$. Moreover, we find that $W \equiv A {\rm tr}h$ on the boundary of $B_r(p)$ where the function ${\rm tr}h$ is uniformly controlled. Then we may assume that $t_0>0$ and $x_0$ does not lie in the boundary of $B_r(p)$. By the maximum principle, we have
\[
0 \leq \left( C_{10}+C_{11}+C_{13}+C_{14}-\frac{A}{N(N+1)} \right) S(x_0,t_0)+C_{13}.
\]
Taking $A:=N(N+1)(C_{10}+C_{11}+C_{13}+C_{14}+1)$, we conclude that $S(x_0,t_0) \leq C_{13}$. Since $0 \leq {\rm tr}h \leq nN$, we have
\[
S \leq \frac{257}{256} C_{13}+AnNK \leq C_{15}
\]
on $\overline{B_{r_1}(p)} \times [0,T]$, where the constant $C_{15}$ depends only on $N$, $\gamma$, $\omega_0$, $X$, $\|\phi(\cdot,0)\|_{C^3(B_r(p))}$ and $\|\widetilde{\eta}\|_{C^1(B_r(p))}$. In particular, $|\nabla_{\phi}X|_{\omega_{\phi}}^2$ is uniformly bounded.

Next, we establish the uniform bound of $|{\rm Rm}_{\phi}|_{\omega_{\phi}}^2$. The evolution equation of the full curvature tensor along MTKRF is
\begin{eqnarray} \label{evc}
\left( \frac{d}{dt}-\Delta_{\omega_{\phi}} \right)R_{\phi \bar{j}i\bar{l}k}&=&R_{\phi \bar{j}i}{}^{p \bar{q}}R_{\phi \bar{l} k \bar{q} p}+R_{\phi \bar{l}i}{}^{p \bar{q}}R_{\phi \bar{j} k \bar{q} p}-R_{\phi \bar{j}p\bar{l}} {}^{\bar{q}} R_{\phi}{}^p {}_{i \bar{q} k}-R_{\phi p \bar{l}}R_{\phi \bar{j}i} {}^p{}_k \nonumber \\
&\hbox{}&-R_{\phi \bar{j} h}R_{\phi}{}^h{}_{i \bar{l} k}-\nabla_{\phi \bar{l}} \nabla_{\phi k} \widetilde{\eta}_{i \bar{j}}+\gamma R_{\phi \bar{j} i \bar{l} k}-\widetilde{\eta}_{\bar{j} h} R_{\phi}{}^h{}_{ik\bar{l}} \nonumber \\
&\hbox{}&\underbrace{-\nabla_{\phi \bar{l}} \nabla_{\phi k} \nabla_{\phi i}X_{\bar{j}}-\nabla_{\phi h}X_{\bar{j}} \cdot R_{\phi}{}^h{}_{ik\bar{l}}}_{\text{additional terms arising from $X$}}.
\end{eqnarray}
By direct computations, we get
\begin{equation} \label{soe}
\nabla_{\phi \bar{l}} \nabla_{\phi k} \widetilde{\eta}_{i \bar{j}}=\nabla_{0 \bar{l}} \nabla_{0 k} \widetilde{\eta}_{i \bar{j}}-U_{\bar{l} \bar{j}}^{\bar{s}} \nabla_{0 k} \widetilde{\eta}_{i \bar{s}}-\nabla_{0 \bar{l}} U_{ki}^s \widetilde{\eta}_{s \bar{j}}-U_{ki}^s\nabla_{0 \bar{l}} \widetilde{\eta}_{s \bar{j}}+U_{ki}^sU_{\bar{l} \bar{j}}^{\bar{t}} \widetilde{\eta}_{s \bar{t}},
\end{equation}
\begin{equation} \label{utc}
\nabla_{0 \bar{k}}U_{jl}^i=\nabla_{\phi \bar{k}}U_{jl}^i=\partial_{\bar{k}}U_{jl}^i=-R_{\phi}{}^i{}_{l \bar{k} j}+R_0{}^i{}_{l \bar{k} j},
\end{equation}
\begin{equation} \label{sao}
\nabla_{\phi \bar{u}} \nabla_{\phi l} \nabla_{\phi j}X^i=-\nabla_{\phi l}X^k \cdot R_{\phi j}{}^i{}_{k \bar{u}}-X^k \nabla_{\phi l} R_{\phi j}{}^i{}_{k \bar{u}}-\nabla_{\phi j}X^p \cdot R_{\phi p} {}^i{}_{l \bar{u}}+\nabla_{\phi s}X^i \cdot R_{\phi l}{}^s{}_{j \bar{u}}.
\end{equation}
Hence, using the uniform bound of $S$, $|X|_{\omega_{\phi}}^2$ and $|\nabla_{\phi} X|_{\omega_{\phi}}^2$, we have
\begin{eqnarray*}
\left| \left( \frac{d}{dt}-\Delta_{\omega_{\phi}} \right){\rm Rm}_{\phi} \right|_{\omega_{\phi}} \leq C_{16}(|{\rm Rm}_{\phi}|_{\omega_{\phi}}^2+|{\rm Rm}_{\phi}|_{\omega_{\phi}}+1)+C_{17}|\nabla_{\phi} {\rm Rm}_{\phi}|_{\omega_{\phi}}.
\end{eqnarray*}
Thus, by the uniform bound of $|\nabla_{\phi} X|_{\omega_{\phi}}^2$ and the equation \eqref{sMAc}, we obtain
\begin{eqnarray} \label{erf}
\left( \frac{d}{dt}-\Delta_{\omega_{\phi}} \right) |{\rm Rm}_{\phi}|_{\omega_{\phi}}^2 &\leq& C_{18}(|{\rm Rm}_{\phi}|_{\omega_{\phi}}^3+|{\rm Rm}_{\phi}|_{\omega_{\phi}}^2)+2\left| \left( \frac{d}{dt}-\Delta_{\omega_{\phi}} \right){\rm Rm}_{\phi} \right|_{\omega_{\phi}} |{\rm Rm}_{\phi}|_{\omega_{\phi}} \nonumber \\
&\hbox{}& -|\nabla_{\phi} {\rm Rm}_{\phi}|_{\omega_{\phi}}^2-|\overline{\nabla}_{\phi} {\rm Rm}_{\phi}|_{\omega_{\phi}}^2 \nonumber \\
&\leq& C_{19}(|{\rm Rm}_{\phi}|_{\omega_{\phi}}^3+1)-\frac{1}{2}|\nabla_{\phi} {\rm Rm}_{\phi}|_{\omega_{\phi}}^2-|\overline{\nabla}_{\phi} {\rm Rm}_{\phi}|_{\omega_{\phi}}^2.
\end{eqnarray}
Now we take a smaller radius $r_2$ satisfying $r_1>r_2>r/2$ and show that $|{\rm Rm}_{\phi}|_{\omega_{\phi}}^2$ is uniformly bounded on $\overline{B_{r_2}(p)}$. Let $\mu$ be a nonnegative smooth cut-off function that is identically equal to $1$ on $\overline{B_{r_2}(p)}$, vanishes on the outside of $B_{r_1}(p)$ and satisfies
\[
|\partial \mu|_{\omega_0}, \;\; |\sqrt{-1}\partial \bar{\partial} \mu |_{\omega_0} \leq C_{20}.
\]
Let $L$ be a uniform constant satisfying $\frac{512}{513}L \leq L-S \leq L$. We consider the function
\[
G:=\mu^2 \frac{|{\rm Rm}_{\phi}|_{\omega_{\phi}}^2}{L-S}+BS,
\]
where $B$ is a uniform constant determined later. By computing, we have
\begin{eqnarray*}
\left( \frac{d}{dt}-\Delta_{\omega_{\phi}} \right)G&=&(-\Delta_{\omega_{\phi}} \mu^2) \frac{|{\rm Rm}_{\phi}|_{\omega_{\phi}}^2}{L-S}-4{\rm Re} \left( \frac{\mu \nabla_{\phi} \mu}{L-S}, \nabla_{\phi}|{\rm Rm}_{\phi}|_{\omega_{\phi}}^2 \right)_{\omega_{\phi}}\\
&\hbox{}&-4{\rm Re} \left( \mu \nabla_{\phi} \mu, \frac{|{\rm Rm}_{\phi}|_{\omega_{\phi}}^2\nabla_{\phi}S}{(L-S)^2} \right)_{\omega_{\phi}}+\frac{\mu^2}{L-S}\left( \frac{d}{dt}-\Delta_{\omega_{\phi}} \right)|{\rm Rm}_{\phi}|_{\omega_{\phi}}^2\\
&\hbox{}&+\frac{\mu^2|{\rm Rm}_{\phi}|_{\omega_{\phi}}^2}{(L-S)^2} \left( \frac{d}{dt}-\Delta_{\omega_{\phi}} \right) S -\frac{2\mu^2 |{\rm Rm}_{\phi}|_{\omega_{\phi}}^2}{(L-S)^3}|\nabla_{\phi}S|_{\omega_{\phi}}^2\\
&\hbox{}&-2{\rm Re} \left( \mu^2 \frac{\nabla_{\phi}S}{(L-S)^2}, \nabla_{\phi}|{\rm Rm}_{\phi}|_{\omega_{\phi}}^2 \right)_{\omega_{\phi}}+B \left( \frac{d}{dt}-\Delta_{\omega_{\phi}} \right) S.
\end{eqnarray*}
Then, by \eqref{eoS}, \eqref{aes}, \eqref{erf} and
\begin{equation}
|\nabla_{\phi} |{\rm Rm}_{\phi}|_{\omega_{\phi}}^2|_{\omega_{\phi}} \leq |{\rm Rm}_{\phi}|_{\omega_{\phi}} (|\nabla_{\phi}{\rm Rm}_{\phi}|_{\omega_{\phi}}+|\overline{\nabla}_{\phi}{\rm Rm}_{\phi}|_{\omega_{\phi}}),
\end{equation}

 we know that
\[
\left| (-\Delta_{\omega_{\phi}} \mu^2) \frac{|{\rm Rm}_{\phi}|_{\omega_{\phi}}^2}{L-S} \right| \leq C_{21} |{\rm Rm}_{\phi}|_{\omega_{\phi}}^2,
\]
\begin{eqnarray*}
\left| 4{\rm Re} \left( \frac{\mu \nabla_{\phi} \mu}{L-S}, \nabla_{\phi}|{\rm Rm}_{\phi}|_{\omega_{\phi}}^2 \right)_{\omega_{\phi}} \right| &\leq& \frac{4}{L-S}\mu |\nabla_{\phi} \mu|_{\omega_{\phi}}|{\rm Rm}|_{\omega_{\phi}}(|\nabla_{\phi} {\rm Rm}_{\phi}|_{\omega_{\phi}}+|\overline{\nabla}_{\phi} {\rm Rm}_{\phi}|_{\omega_{\phi}})\\
&\leq&C_{22}|{\rm Rm}_{\phi}|_{\omega_{\phi}}^2+\frac{\mu^2}{4(L-S)}(|\nabla_{\phi} {\rm Rm}_{\phi}|_{\omega_{\phi}}^2+|\overline{\nabla}_{\phi} {\rm Rm}_{\phi}|_{\omega_{\phi}}^2),
\end{eqnarray*}
\begin{eqnarray*}
\left| 4{\rm Re} \left( \mu \nabla_{\phi} \mu, \frac{|{\rm Rm}_{\phi}|_{\omega_{\phi}}^2\nabla_{\phi}S}{(L-S)^2} \right)_{\omega_{\phi}} \right| &\leq& \frac{4 \sqrt{2} |{\rm Rm}_{\phi}|_{\omega_{\phi}}^2}{(L-S)^2} \mu |\nabla_{\phi} \mu|_{\omega_{\phi}} S^{1/2} (|\nabla_{\phi}U|_{\omega_{\phi}}^2+|\overline{\nabla}_{\phi}U|_{\omega_{\phi}}^2)^{1/2}\\
&\leq& C_{23}|{\rm Rm}_{\phi}|_{\omega_{\phi}}^2+\frac{\mu^2 |{\rm Rm}_{\phi}|_{\omega_{\phi}}^2}{4(L-S)^2}(|\nabla_{\phi}U|_{\omega_{\phi}}^2+|\overline{\nabla}_{\phi}U|_{\omega_{\phi}}^2),
\end{eqnarray*}
\begin{eqnarray*}
\frac{\mu^2}{L-S} \left( \frac{d}{dt}-\Delta_{\omega_{\phi}} \right)|{\rm Rm}_{\phi}|_{\omega_{\phi}}^2 &\leq& \frac{C_{19}\mu^2}{L-S}|{\rm Rm}_{\phi}|_{\omega_{\phi}}^3-\frac{\mu^2}{2(L-S)}(|\nabla_{\phi} {\rm Rm}_{\phi}|_{\omega_{\phi}}^2+|\overline{\nabla}_{\phi} {\rm Rm}_{\phi}|_{\omega_{\phi}}^2)+C_{24}\\
&\leq& \frac{\mu^2|{\rm Rm_{\phi}}|_{\omega_{\phi}}^4}{8(L-S)^2}+C_{25} \mu^2 |{\rm Rm_{\phi}}|_{\omega_{\phi}}^2-\frac{\mu^2}{2(L-S)}(|\nabla_{\phi} {\rm Rm}_{\phi}|_{\omega_{\phi}}^2+|\overline{\nabla}_{\phi} {\rm Rm}_{\phi}|_{\omega_{\phi}}^2)\\
&\hbox{}&+C_{24}\\
&\leq& C_{26}|{\rm Rm}_{\phi}|_{\omega_{\phi}}^2+\frac{\mu^2|{\rm Rm}_{\phi}|_{\omega_{\phi}}^2}{8(L-S)^2}(|\nabla_{\phi}U|_{\omega_{\phi}}^2+|\overline{\nabla}_{\phi}U|_{\omega_{\phi}}^2)\\
&\hbox{}& -\frac{\mu^2}{2(L-S)}(|\nabla_{\phi} {\rm Rm}_{\phi}|_{\omega_{\phi}}^2+|\overline{\nabla}_{\phi} {\rm Rm}_{\phi}|_{\omega_{\phi}}^2)+C_{24}\\
&\hbox{}& \text{(where we used \eqref{utc} in the last inequality)},
\end{eqnarray*}
\[
\frac{\mu^2|{\rm Rm}_{\phi}|_{\omega_{\phi}}^2}{(L-S)^2} \left( \frac{d}{dt}-\Delta_{\omega_{\phi}} \right) S \leq C_{27}|{\rm Rm}_{\phi}|_{\omega_{\phi}}^2-\frac{\mu^2|{\rm Rm}_{\phi}|_{\omega_{\phi}}^2}{2(L-S)^2}(|\nabla_{\phi}U|_{\omega_{\phi}}^2+|\overline{\nabla}_{\phi}U|_{\omega_{\phi}}^2),
\]
\begin{eqnarray*}
\left| 2{\rm Re} \left( \mu^2 \frac{\nabla_{\phi}S}{(L-S)^2}, \nabla_{\phi}|{\rm Rm}_{\phi}|_{\omega_{\phi}}^2 \right)_{\omega_{\phi}} \right| &\leq& \frac{2 \sqrt{2} \mu^2}{(L-S)^2} S^{1/2}(|\nabla_{\phi}U|_{\omega_{\phi}}^2+|\overline{\nabla}_{\phi}U| _{\omega_{\phi}}^2)^{1/2} \cdot \\
&\hbox{}& |{\rm Rm}_{\phi}|_{\omega_{\phi}}(|\nabla_{\phi} {\rm Rm}_{\phi}|_{\omega_{\phi}}+|\overline{\nabla}_{\phi} {\rm Rm}_{\phi}|_{\omega_{\phi}})\\
&\leq& \frac{64\mu^2S}{(L-S)^2}(|\nabla_{\phi} {\rm Rm}_{\phi}|_{\omega_{\phi}}^2+|\overline{\nabla}_{\phi} {\rm Rm}_{\phi}|_{\omega_{\phi}}^2)\\
&\hbox{}&+\frac{\mu^2|{\rm Rm}_{\phi}|_{\omega_{\phi}}^2}{16(L-S)^2}(|\nabla_{\phi}U|_{\omega_{\phi}}^2+|\overline{\nabla}_{\phi}U| _{\omega_{\phi}}^2)\\
&\leq& \frac{\mu^2}{8(L-S)}(|\nabla_{\phi} {\rm Rm}_{\phi}|_{\omega_{\phi}}^2+|\overline{\nabla}_{\phi} {\rm Rm}_{\phi}|_{\omega_{\phi}}^2)\\
&\hbox{}&+\frac{\mu^2|{\rm Rm}_{\phi}|_{\omega_{\phi}}^2}{16(L-S)^2}(|\nabla_{\phi}U|_{\omega_{\phi}}^2+|\overline{\nabla}_{\phi}U| _{\omega_{\phi}}^2)\\
&\hbox{}& \text{(because $\frac{512}{513}L < L-S<L$)}.
\end{eqnarray*}
As in the previous part, we may only consider an inner point $(x_0,t_0)$ which is a maximum point of $G$ achieved on $\overline{B_{r_1}(p)} \times [0,T]$. By the maximum principle, we have
\[
0 \leq \left( C_{21}+C_{22}+C_{23}+C_{26}+C_{27}-\frac{B}{2} \right)|{\rm Rm}_{\phi}|_{\omega_{\phi}}^2(x_0,t_0)+C_{28}.
\]
Now we set $B:=2(C_{21}+C_{22}+C_{23}+C_{26}+C_{27}+1)$. Then we obtain
\[
|{\rm Rm}_{\phi}|_{\omega_{\phi}}^2(x_0,t_0) \leq C_{28}.
\]
Since $S$ is uniformly bounded, this implies
\[
|{\rm Rm}_{\phi}|_{\omega_{\phi}}^2 \leq C_{29}
\]
on $\overline{B_{r_2}(p)} \times [0,T]$, where $C_{29}$ depends only on $N$, $\gamma$, $\omega_0$, $X$, $\|\phi(\cdot,0)\|_{C^4(B_r(p))}$ and $\| \widetilde{\eta} \|_{C^2(B_r(p))}$.

Following \cite{LZ17}, we say that $\phi$ is $C^{k,\alpha}$ if its $C^{k,\alpha}$ norm can be controlled by a constant depending only on $N$, $\gamma$, $\omega_0$, $X$, $\|\phi(\cdot,0)\|_{C^{k+1}(B_r(p))}$, $\|\phi\|_{C^0(B_r(p)\times[0,T])}$, $\|\widetilde{\eta}\|_{C^{k-1}(B_r(p))}$ and $\|F\|_{C^0(B_r(p))}$. Likewise, we say that $\dot{\phi}$ is $C^{k,\alpha}$ if its $C^{k,\alpha}$ norm can be controlled by a constant depending only on $N$, $\gamma$, $\omega_0$, $X$, $\|\phi(\cdot,0)\|_{C^{k+3}(B_r(p))}$, $\|\phi\|_{C^0(B_r(p)\times[0,T])}$, $\|\widetilde{\eta}\|_{C^{k+1}(B_r(p))}$ and $\|F\|_{C^0(B_r(p))}$. Since $|{\rm Rm}_{\phi}|_{\omega_{\phi}}^2$ and $|\nabla_{\phi}X|_{\omega_{\phi}}^2$ are uniformly bounded, we know that $\dot{\phi}$ is $C^{1,\alpha}$. Differentiating the equation \eqref{MAf} with respect to $z^k$, we get
\[
\frac{d}{dt}\frac{\partial \phi}{\partial z^k}=(\Delta_{\omega_{\phi}}+X)\frac{\partial \phi}{\partial z^k}+g_{\phi}^{i\bar{j}}\frac{\partial g_{0 i \bar{j}}}{\partial z^k}-g_0^{i \bar{j}}\frac{g_{0 i \bar{j}}}{\partial z^k}+\frac{\partial F}{\partial z^k}+\gamma \frac{\partial \phi}{\partial z^k}+\frac{\partial \theta_X}{\partial z^k}+\frac{\partial X^i}{\partial z^k} \frac{\partial \phi}{\partial z^i}.
\]
From the above Calabi's $C^3$-estimate, we know that $\phi$ is $C^{2,\alpha}$ and then the coefficients of $\Delta_{\omega_{\phi}}$ are $C^{0,\alpha}$. Since $F$ is the twisted Ricci potential, taking the trace with respect to $\omega_0$ yields
\[
\Delta_{\omega_0}F=-{\rm tr}_{\omega_0}{\rm Ric}(\omega_0)+\gamma+{\rm tr}_{\omega_0}\widetilde{\eta}.
\]
Hence the $C^{1,\alpha}$-norm of $F$ on $B_{r_2}(p)$ only depends on $\omega_0$, $\|\widetilde{\eta}\|_{C^0(B_r(p))}$ and $\|F\|_{C^0(B_r(p))}$. By the standard elliptic Schauder estimates, we conclude that $\phi$ is $C^{3,\alpha}$ on $B_{r_3}(p) \times [0,T]$, where $r_2>r_3>r/2$.

Now we prove that $|\nabla_{\phi} {\rm Rm}_{\phi}|_{\omega_{\phi}}^2$ is uniformly bounded. First we compute the evolution equation of $U$ as
\begin{equation} \label{eut}
\left( \frac{d}{dt}-\Delta_{\omega_{\phi}} \right) U_{ml}^{\beta}=\nabla_{\phi m}(\widetilde{\eta}^{\beta}{}_l+\nabla_{\phi l}X^{\beta})-\nabla_{\phi}^{\bar{q}}R_0{}^{\beta}{}_{l \bar{q}m}.
\end{equation}
Since $\widetilde{\eta}$, ${\rm Rm}_0$ and $X$ are $t$-independent tensors, we know that
\begin{equation} \label{uce}
|\nabla_{\phi} \widetilde{\eta}|_{\omega_{\phi}} \leq C_{30},
\end{equation}
\begin{equation} \label{ucv}
|\nabla_{\phi}^2 \widetilde{\eta}|_{\omega_{\phi}}+|\nabla_{\phi}^2 {\rm Rm}_0|_{\omega_{\phi}}+|\nabla_{\phi}^2 X|_{\omega_{\phi}} \leq C_{31}(1+|\nabla_{\phi} U|_{\omega_{\phi}}),
\end{equation}
\[
|\nabla_{\phi}^3X|_{\omega_{\phi}} \leq C_{32}(1+|\nabla_{\phi}U|_{\omega_{\phi}}+|\nabla_{\phi}^2U|_{\omega_{\phi}}).
\]
On the other hand, by the Ricci identity, we have
\[
\left( \frac{d}{dt}-\Delta_{\omega_{\phi}} \right) \nabla_{\phi} U =\nabla_{\phi} \left( \frac{d}{dt}-\Delta_{\omega_{\phi}} \right) U+U \ast \nabla_{\phi}({\rm Rm}_{\phi}+\widetilde{\eta}+\nabla_{\phi} X)+{\rm Rm}_{\phi} \ast \nabla_{\phi}U,
\]
where $\ast$ means the general pairs of tensors. Thus we obtain
\begin{equation} \label{eou}
\left( \frac{d}{dt}-\Delta_{\omega_{\phi}} \right) |\nabla_{\phi}U|_{\omega_{\phi}}^2 \leq C_{33}(|\nabla_{\phi}U|_{\omega_{\phi}}^2+1)+|\nabla_{\phi} {\rm Rm}_{\phi}|_{\omega_{\phi}}^2-\frac{1}{2}|\nabla_{\phi} \nabla_{\phi} U|_{\omega_{\phi}}^2-|\overline{\nabla}_{\phi} \nabla_{\phi}U|_{\omega_{\phi}}^2.
\end{equation}
Now we set $r_3 >r'_3>r/2$ and take a smooth cut-off function $\varrho$ such that
\[
|\partial \varrho|_{\omega_0}, \;\; |\sqrt{-1}\partial \bar{\partial} \varrho |_{\omega_0} \leq C_{34},
\]
and set
\[
I:=\varrho^2 |\nabla_{\phi} U|_{\omega_{\phi}}^2+ES+2|{\rm Rm}_{\phi}|_{\omega_{\phi}}^2,
\]
where $E$ is a uniform constant determined later. Then we see that
\begin{eqnarray*}
\left( \frac{d}{dt}-\Delta_{\omega_{\phi}} \right) I &\leq& (-\Delta_{\omega_{\phi}} \varrho^2)|\nabla_{\phi} U|_{\omega_{\phi}}^2-4{\rm Re}(\varrho \nabla_{\phi} \varrho, \nabla_{\phi}|\nabla_{\phi}U|_{\omega_{\phi}}^2)_{\omega_{\phi}}\\
&\hbox{}&+\varrho^2 \left( \frac{d}{dt}-\Delta_{\omega_{\phi}} \right) |\nabla_{\phi}U|_{\omega_{\phi}}^2+E \left( \frac{d}{dt}-\Delta_{\omega_{\phi}} \right) S+2 \left( \frac{d}{dt}-\Delta_{\omega_{\phi}} \right) |{\rm Rm}_{\phi}|_{\omega_{\phi}}^2.
\end{eqnarray*}
The first and second term of the RHS are estimated as
\[
|(-\Delta_{\omega_{\phi}} \varrho^2)|\nabla_{\phi} U|_{\omega_{\phi}}^2| \leq C_{35} |\nabla_{\phi} U|_{\omega_{\phi}}^2,
\]
\[
|4{\rm Re}(\varrho \nabla_{\phi} \varrho, \nabla_{\phi}|\nabla_{\phi}U|_{\omega_{\phi}}^2)_{\omega_{\phi}}| \leq C_{36}|\nabla_{\phi}U|_{\omega_{\phi}}^2+\frac{\varrho^2}{4}(|\nabla_{\phi}\nabla_{\phi}U|_{\omega_{\phi}}^2+|\overline{\nabla}_{\phi}\nabla_{\phi}U|_{\omega_{\phi}}^2).
\]
Thus, combining with \eqref{eoS} and \eqref{erf}, we obtain
\[
\left( \frac{d}{dt}-\Delta_{\omega_{\phi}} \right) I \leq \left( C_{33}+C_{35}+C_{36}-\frac{E}{2} \right) |\nabla_{\phi}U|_{\omega_{\phi}}^2+C_{37}.
\]
Hence, if we set $E:=2(C_{33}+C_{35}+C_{36}+1)$, the maximum principle implies the uniform bound of $|\nabla_{\phi}U|_{\omega_{\phi}}^2$ on $\overline{B_{r'_3}(p)} \times [0,T]$. Let $D$ denote the real covariant derivative with respect to $\omega_{\phi}$ (extended linearly on the space of complex tensors). Combining with the uniform bound of $|{\rm Rm}_{\phi}|_{\omega_{\phi}}^2$ and \eqref{utc}, we have
\[
|DU|_{\omega_{\phi}}^2 \leq C_{38}
\]
on $\overline{B_{r'_3}(p)} \times [0,T]$, where the constant $C_{38}$ depends only on $N$, $\gamma$, $\omega_0$, $X$, $\| \phi(\cdot, 0)\|_{C^4(B_r(P))}$ and $\| \widetilde{\eta} \|_{C^2(B_r(p))}$. In particular, we find that $|D^2X|_{\omega_{\phi}}^2$ is uniformly bounded. Applying $\nabla_{\phi}$ to \eqref{evc}, we see that
\[
\left| \nabla_{\phi} \left( \frac{d}{dt}-\Delta_{\omega_{\phi}} \right) {\rm Rm}_{\phi} \right|_{\omega_{\phi}} \leq C_{39}(|\nabla_{\phi}{\rm Rm}_{\phi}|_{\omega_{\phi}}+|\nabla_{\phi} \overline{\nabla}_{\phi}\nabla_{\phi} \widetilde{\eta}|_{\omega_{\phi}}+|\nabla_{\phi}\overline{\nabla}_{\phi}\nabla_{\phi}^2 X|_{\omega_{\phi}}).
\]
Applying $\nabla_{\phi}$ to \eqref{soe} and \eqref{sao}, and using the uniform bound of $|DU|_{\omega_{\phi}}^2$, we have
\[
|\nabla_{\phi} \overline{\nabla}_{\phi} \nabla_{\phi} \widetilde{\eta}|_{\omega_{\phi}} \leq C_{40}(1+|\nabla_{\phi} {\rm Rm}_{\phi}|_{\omega_{\phi}}),
\]
\[
|\nabla_{\phi} \overline{\nabla}_{\phi} \nabla_{\phi}^2 X|_{\omega_{\phi}} \leq C_{41}(1+|\nabla_{\phi} {\rm Rm}_{\phi}|_{\omega_{\phi}}+|\nabla_{\phi}^2 {\rm Rm}_{\phi}|_{\omega_{\phi}}).
\]
Combining with
\[
\left( \frac{d}{dt}-\Delta_{\omega_{\phi}} \right) \nabla_{\phi}{\rm Rm}_{\phi} =\nabla_{\phi} \left(\frac{d}{dt}-\Delta_{\omega_{\phi}} \right) {\rm Rm}_{\phi}+{\rm Rm}_{\phi} \ast \nabla_{\phi} ({\rm Rm}_{\phi}+\widetilde{\eta}+\nabla_{\phi}X),
\]
we find that
\[
\left( \frac{d}{dt}-\Delta_{\omega_{\phi}} \right) |\nabla_{\phi} {\rm Rm}_{\phi}|_{\omega_{\phi}}^2 \leq C_{42}(|\nabla_{\phi} {\rm Rm}_{\phi}|_{\omega_{\phi}}^2+1)-\frac{1}{2}|\nabla_{\phi} \nabla_{\phi} {\rm Rm}_{\phi}|_{\omega_{\phi}}^2-|\overline{\nabla}_{\phi} \nabla_{\phi} {\rm Rm}_{\phi}|_{\omega_{\phi}}^2.
\]
Now we take a smaller radius $r'_3>r''_3>r/2$ and a smooth cut-off function $\sigma$ that is identically equal to $1$ on $\overline{B_{r''_3}(p)}$, vanishes on the outside of $B_{r'_3}(p)$ and satisfies
\[
|\partial \sigma|_{\omega_0}, \;\; |\sqrt{-1}\partial \bar{\partial} \sigma |_{\omega_0} \leq C_{43}.
\]
We apply the maximum principle to the function $\sigma^2 |\nabla_{\phi} {\rm Rm}_{\phi}|_{\omega_{\phi}}^2+P|{\rm Rm}_{\phi}|_{\omega_{\phi}}^2$ (where $P$ is a suitable uniform constant). Then, as in the previous argument, we find that $|\nabla_{\phi} {\rm Rm}_{\phi}|_{\omega_{\phi}}^2$ is uniformly bounded on $\overline{B_{r''_3}(p)} \times [0,T]$. Thus we have
\[
|D {\rm Rm}_{\phi}|_{\omega_{\phi}}^2 \leq C_{44}
\]
on $\overline{B_{r''_3}(p)} \times [0,T]$, where $C_{44}$ depends only on $N$, $\gamma$, $\omega_0$, $X$, $\|\phi(\cdot,0)\|_{C^5(B_r(p))}$ and $\| \widetilde{\eta} \|_{C^3(B_r(p))}$.

Applying $D$ to the equation \eqref{sMAc}, we have
\[
D \sqrt{-1} \partial \bar{\partial} \dot{\phi} =D {\rm Ric}(\omega_{\phi})+D \widetilde{\eta}+D(\nabla_{\phi} X^{\flat}),
\]
where $X^{\flat}_{\bar{j}}:=g_{\phi i \bar{j}} X^i$. Taking the trace, we have
\begin{eqnarray*}
|\Delta_{\omega_{\phi}}D \dot{\phi}|_{\omega_{\phi}} &\leq& |D \Delta_{\omega_{\phi}} \dot{\phi} |_{\omega_{\phi}}+|D {\rm Rm}_{\phi} \ast \dot{\phi}|_{\omega_{\phi}}+|{\rm Rm}_{\phi} \ast D \dot{\phi}|_{\omega_{\phi}}\\
&\leq& C_{45}(|D{\rm Rm}_{\phi}|_{\omega_{\phi}}+|D \widetilde{\eta}|_{\omega_{\phi}}+|D^2 X|_{\omega_{\phi}}+|D{\rm Rm}_{\phi}|_{\omega_{\phi}}|\dot{\phi}|+|{\rm Rm}_{\phi}|_{\omega_{\phi}} |D \dot{\phi}|_{\omega_{\phi}})
\end{eqnarray*}
From the above computations and the fact that $\dot{\phi}$ is $C^{1,\alpha}$, we find that $D \dot{\phi}$ is $C^{1,\alpha}$, which implies that $\dot{\phi}$ is $C^{2,\alpha}$. Differentiating the equation \eqref{MAf} two times and using the elliptic Schauder estimates, we have $\phi$ is $C^{4,\alpha}$ on $B_{r_4}(p) \times [0,T]$, where $r''_3>r_4>r/2$.

Now we establish the $C^{k,\alpha}$-estimate for $\phi$. For this, we set the following induction hypothesis:
\[
\begin{cases}
|D^j {\rm Rm}|_{\omega_{\phi}}^2 \leq C_j^1 \\
\text{$\dot{\phi}$ is $C^{j+1,\alpha}$} \\
\text{$\phi$ is $C^{j+3,\alpha}$}
\end{cases}
\text{on $\overline{B_{r_{j+3}}(p)} \times [0,T]$ for all $j=0,1, \ldots k$}, \leqno{(H_k)}
\]
where $r>r_1>\cdots>r_{k+2}>r_{k+3}>r/2$ and the constant $C_j^1$ depends only on $N$, $\gamma$, $\omega_0$, $X$, $\|\phi(\cdot,0)\|_{C^{j+4}(B_r(p))}$, $\|\phi \|_{C^0(B_r(p) \times [0,T])}$, $\|\widetilde{\eta}\|_{C^{j+2}(B_r(p))}$ and $\|F\|_{C^0(B_r(p))}$. We have already seen that this statement is established for $k=0,1$. Now we assume that the induction hypothesis $(H_k)$ holds for some $k \geq 1$. Since $\phi$ is $C^{k+3, \alpha}$, we observe that
\[
|D^j U|_{\omega_{\phi}}^2 \leq C_{46} \;\; \text{for $j=0,1,\ldots,k$}.
\]
In particular, for any $t$-independent tensor $A$, we find that $|D^jA|_{\omega_{\phi}}^2$ is uniformly bounded for $j=0,1,\ldots,k+1$.
We first show the uniform bound of $|D^{k+1}U|_{\omega_{\phi}}^2$. Let $r,s$ ($r+s=k+1$) are non-negative integers. Then any $(k+1)$-detivative of $U$ differs from $\nabla_{\phi}^r \overline{\nabla}_{\phi}^s U$ by a linear combination of $D^iU \ast D^{r+s-2-i}{\rm Rm}_{\phi}$ ($0 \leq i \leq r+s-2$), which has been already estimated by the induction hypothesis $(H_k)$. Thus we may only consider $\nabla_{\phi}^r \overline{\nabla}_{\phi}^s U$. Moreover, the equation \eqref{utc} and $(H_k)$ indicate that we should only consider $\nabla_{\phi}^{k+1} U$. Using the Ricci identity repeatedly, we have
\begin{eqnarray*}
\left( \frac{d}{dt}-\Delta_{\omega_{\phi}} \right) \nabla_{\phi}^{k+1}U&=&\underbrace{\nabla_{\phi}^{k+1} \left( \frac{d}{dt}-\Delta_{\omega_{\phi}} \right) U}_\text{($\nabla^{k+1}U$;I)}+\underbrace{\sum_{\substack{p \geq 0, q \geq 1\\ p+q=k+1}} \nabla_{\phi}^p U \ast \nabla_{\phi}^q ({\rm Rm}_{\phi}+\widetilde{\eta}+\nabla_{\phi}X)}_\text{($\nabla^{k+1}U$;II)}\\
&+& \underbrace{\sum_{\substack{p \geq 0, q \geq 1\\ p+q=k+1}} \nabla_{\phi}^p{\rm Rm}_{\phi} \ast \nabla_{\phi}^q U}_\text{($\nabla^{k+1}U$;III)}.
\end{eqnarray*}
By \eqref{eut} and $(H_k)$, we observe that
\[
|\text{($\nabla^{k+1}U$;I)}|_{\omega_{\phi}} \leq C_{47}(1+|\nabla_{\phi}^{k+1}U|_{\omega_{\phi}}+|\nabla_{\phi}^{k+2}U|_{\omega_{\phi}}),
\]
\[
|\text{($\nabla^{k+1}U$;II)}|_{\omega_{\phi}} \leq C_{48}(1+|\nabla_{\phi}^{k+1}{\rm Rm}_{\phi}|_{\omega_{\phi}}+|\nabla_{\phi}^{k+1}U|_{\omega_{\phi}}),
\]
\[
|\text{($\nabla^{k+1}U$;III)}|_{\omega_{\phi}} \leq C_{49}(1+|\nabla_{\phi}^{k+1}{\rm Rm}_{\phi}|_{\omega_{\phi}}).
\]
Thus the evolution equation of $|\nabla_{\phi}^{k+1}U|_{\omega_{\phi}}^2$ can be estimated as
\begin{equation} \label{huf}
\left( \frac{d}{dt}-\Delta_{\omega_{\phi}} \right)|\nabla_{\phi}^{k+1}U|_{\omega_{\phi}}^2 \leq -\frac{1}{2}|\nabla_{\phi}^{k+2}U|_{\omega_{\phi}}^2-|\overline{\nabla}_{\phi}\nabla_{\phi}^{k+1}U|_{\omega_{\phi}}^2+C_{50}|\nabla_{\phi}^{k+1}U|_{\omega_{\phi}}^2+|\nabla_{\phi}^{k+1}{\rm Rm}_{\phi}|_{\omega_{\phi}}^2.
\end{equation}
Hence we should compute the evolution equation of $|\nabla_{\phi}^kU|_{\omega_{\phi}}^2$ and $|\nabla_{\phi}^k{\rm Rm}_{\phi}|_{\omega_{\phi}}^2$, and add them to the above equation. It is not hard to see that
\begin{equation}
\left( \frac{d}{dt}-\Delta_{\omega_{\phi}} \right) |\nabla_{\phi}^k U|_{\omega_{\phi}}^2 \leq C_{51} -\frac{1}{2}|\nabla_{\phi}^{k+1}U|_{\omega_{\phi}}^2-|\overline{\nabla}_{\phi}\nabla_{\phi}^k U|_{\omega_{\phi}}^2,
\end{equation}
\begin{equation} \label{emc}
\left( \frac{d}{dt}-\Delta_{\omega_{\phi}} \right) |\nabla_{\phi}^k{\rm Rm}_{\phi}|_{\omega_{\phi}}^2 \leq C_{52}-\frac{1}{2}|\nabla_{\phi}^{k+1}{\rm Rm}_{\phi}|_{\omega_{\phi}}^2-|\overline{\nabla}_{\phi} \nabla_{\phi}^k{\rm Rm}_{\phi}|_{\omega_{\phi}}^2.
\end{equation}
Actually, we can compute the first item in the same way as \eqref{huf}. For the second item, one should refer to the computation of \eqref{epc}.
Hence we take a smooth cut-off function $\varsigma$ and apply the maximum principle to the function $\varsigma^2|\nabla_{\phi}^{k+1} U|_{\omega_{\phi}}^2+Q|\nabla_{\phi}^k U|_{\omega_{\phi}}^2+2|\nabla_{\phi}^k{\rm Rm}_{\phi}|_{\omega_{\phi}}^2$ (for a suitable uniform constant $Q$) to get the uniform control of $|\nabla_{\phi}^{k+1} U|_{\omega_{\phi}}^2$ in $\overline{B_{r'_{j+3}}(p)} \times [0,T]$ with a smaller radius $r_{k+3}>r'_{k+3}>r/2$. Thus we have
\[
|D^{k+1} U|_{\omega_{\phi}}^2 \leq C_{53}
\]
on $\overline{B_{r'_{j+3}}(p)} \times [0,T]$, where the constant $C_{53}$ depends only on $N$, $\gamma$, $\omega_0$, $X$, $\|\phi(\cdot, 0)\|_{C^{k+4}(B_r (p))}$, $\|\phi\|_{C^0(B_r(p) \times [0,T])}$, $\|\widetilde{\eta}\|_{C^{k+2}(B_r(p))}$ and $\|F\|_{C^0(B_r(p))}$. In particular, we find that $|D^{k+2}X|_{\omega_{\phi}}^2$ is uniformly bounded.

Next, we establish the uniform estimate for $|D^{k+1} {\rm Rm}_{\phi}|_{\omega_{\phi}}^2$. As in the previous case, we may only consider the tensor of the form $\nabla_{\phi}^r \overline{\nabla}_{\phi}^s {\rm Rm}_{\phi}$ for non-negative integers $r$, $s$ such that $r+s=k+1$. Moreover, by the symmetries of ${\rm Rm}_{\phi}$, we may also assume that $r \neq 0$.

\vspace{4mm}
\noindent
\textbf{Case 1:} $r, s \neq 0$.

Using the Ricci identity repeatedly, we have
\begin{eqnarray*}
\left( \frac{d}{dt}-\Delta_{\omega_{\phi}} \right) \nabla_{\phi}^r \overline{\nabla}_{\phi}^s {\rm Rm}_{\phi}&=&\underbrace{\nabla_{\phi}^r \overline{\nabla}_{\phi}^s {\rm Rm}_{\phi} \left( \frac{d}{dt}-\Delta_{\omega_{\phi}} \right) {\rm Rm}_{\phi}}_\text{($\nabla^r\overline{\nabla}^s{\rm Rm}$;I)}\\
&\hbox{}&+\underbrace{\sum_{\substack{p \geq 0, q \geq 1\\ p+q=k+1}} \nabla_{\phi}^p \overline{\nabla}_{\phi}^s{\rm Rm}_{\phi} \ast \nabla_{\phi}^q({\rm Rm}_{\phi}+\widetilde{\eta}+\nabla_{\phi}X)}_\text{($\nabla^r\overline{\nabla}^s{\rm Rm}$;II)}\\
&\hbox{}&+\underbrace{\sum_{\substack{p \geq 0, q \geq 1\\ p+q=k+1}} \nabla_{\phi}^p {\rm Rm}_{\phi} \ast \nabla_{\phi}^q \overline{\nabla}_{\phi}^s({\rm Rm}_{\phi}+\widetilde{\eta}+\nabla_{\phi}X)}_\text{($\nabla^r\overline{\nabla}^s{\rm Rm}$;III)}\\
&\hbox{}&+\underbrace{\sum_{\substack{p \geq 0, q \geq 1\\ p+q=k+1\\ i=0,1,\ldots,r}} \nabla_{\phi}^i \overline{\nabla}_{\phi}^p {\rm Rm}_{\phi} \ast \nabla_{\phi}^{r-i} \overline{\nabla}_{\phi}^q ({\rm Rm}_{\phi}+\widetilde{\eta}+\nabla_{\phi}X)}_\text{($\nabla^r\overline{\nabla}^s{\rm Rm}$;IV)}\\
&\hbox{}&+\underbrace{\sum_{\substack{p \geq 0, q \geq 1\\ p+q=k+1\\ i=0,1,\ldots,r}} \nabla_{\phi}^i \overline{\nabla}_{\phi}^p {\rm Rm}_{\phi} \ast \nabla_{\phi}^{r-i} \overline{\nabla}_{\phi}^q {\rm Rm}_{\phi}}_\text{($\nabla^r\overline{\nabla}^s{\rm Rm}$;V)}.
\end{eqnarray*}
By \eqref{evc}, \eqref{soe}, \eqref{utc}, \eqref{sao} and the uniform bound of $|D^{k+1}U|_{\omega_{\phi}}^2$, we can estimate the first term as follows:
\begin{eqnarray*}
|\text{($\nabla^r\overline{\nabla}^s{\rm Rm}$;I)}|_{\omega_{\phi}} &=&\nabla_{\phi}^r \overline{\nabla}_{\phi}^s({\rm Rm}_{\phi} \ast {\rm Rm}_{\phi}+\overline{\nabla}_{\phi} \nabla_{\phi} \widetilde{\eta}+{\rm Rm}_{\phi}+\widetilde{\eta} \ast {\rm Rm}_{\phi}+\overline{\nabla}_{\phi} \nabla_{\phi}^2 X+\nabla_{\phi}X \ast {\rm Rm}_{\phi})\\
&\leq& C_{54}(1+|\nabla_{\phi}^r \overline{\nabla}_{\phi}^s({\rm Rm}_{\phi} \ast {\rm Rm}_{\phi})|_{\omega_{\phi}}+|\nabla_{\phi}^r \overline{\nabla}_{\phi}^{s+1} \nabla_{\phi} \widetilde{\eta}|_{\omega_{\phi}}+|\nabla_{\phi}^r \overline{\nabla}_{\phi}^{s+1} \nabla_{\phi}^2 X|_{\omega_{\phi}}),
\end{eqnarray*}
\[
|\nabla_{\phi}^r \overline{\nabla}_{\phi}^s({\rm Rm}_{\phi} \ast {\rm Rm}_{\phi})|_{\omega_{\phi}}+|\nabla_{\phi}^r \overline{\nabla}_{\phi}^{s+1} \nabla_{\phi} \widetilde{\eta}|_{\omega_{\phi}} \leq C_{55}(1+|\nabla_{\phi}^r \overline{\nabla}_{\phi}^s {\rm Rm}_{\phi}|_{\omega_{\phi}}),
\]
\begin{eqnarray*}
|\nabla_{\phi}^r \overline{\nabla}_{\phi}^{s+1} \nabla_{\phi}^2 X|_{\omega_{\phi}} &\leq& C_{56} (1+|\nabla_{\phi}^r \overline{\nabla}_{\phi}^s {\rm Rm}_{\phi}|_{\omega_{\phi}}+|\nabla_{\phi}^r \overline{\nabla}_{\phi}^s \nabla_{\phi} {\rm Rm}_{\phi}|_{\omega_{\phi}}) \\
&\leq& C_{57}(1+|\nabla_{\phi}^r \overline{\nabla}_{\phi}^s {\rm Rm}_{\phi}|_{\omega_{\phi}}+|\nabla_{\phi}^{r+1} \overline{\nabla}_{\phi}^s {\rm Rm}_{\phi}|_{\omega_{\phi}})\\
&\hbox{}& (\text{where we used the Ricci identity and $(H_k)$}).
\end{eqnarray*}
Other terms are easier and estimated as follows:
\[
|\text{($\nabla^r\overline{\nabla}^s{\rm Rm}$;II)}|_{\omega_{\phi}}+|\text{($\nabla^r\overline{\nabla}^s{\rm Rm}$;III)}|_{\omega_{\phi}} \leq C_{58},
\]
\[
|\text{($\nabla^r\overline{\nabla}^s{\rm Rm}$;IV)}|_{\omega_{\phi}}+|\text{($\nabla^r\overline{\nabla}^s{\rm Rm}$;V)}|_{\omega_{\phi}} \leq C_{59}(1+|\nabla_{\phi}^r \overline{\nabla}_{\phi}^s {\rm Rm}_{\phi}|_{\omega_{\phi}}).
\]
Hence we have
\begin{equation}
\left( \frac{d}{dt}-\Delta_{\omega_{\phi}} \right) |\nabla_{\phi}^r \overline{\nabla}_{\phi}^s {\rm Rm}_{\phi}|_{\omega_{\phi}}^2 \leq C_{60}|\nabla_{\phi}^r \overline{\nabla}_{\phi}^s {\rm Rm}_{\phi}|_{\omega_{\phi}}^2-\frac{1}{2}|\nabla_{\phi}^{r+1} \overline{\nabla}_{\phi}^s {\rm Rm}_{\phi}|_{\omega_{\phi}}^2-|\overline{\nabla}_{\phi}\nabla_{\phi}^r \overline{\nabla}_{\phi}^s {\rm Rm}_{\phi}|_{\omega_{\phi}}^2.
\end{equation}
We can estimate the evolution equation of $|\nabla_{\phi}^{r-1} \overline{\nabla}_{\phi}^s {\rm Rm}_{\phi}|_{\omega_{\phi}}^2$ in a similar way to get
\begin{equation}
\left( \frac{d}{dt}-\Delta_{\omega_{\phi}} \right) |\nabla_{\phi}^{r-1} \overline{\nabla}_{\phi}^s {\rm Rm}_{\phi}|_{\omega_{\phi}}^2 \leq C_{61}-\frac{1}{2}|\nabla_{\phi}^r \overline{\nabla}_{\phi}^s {\rm Rm}_{\phi}|_{\omega_{\phi}}^2-|\overline{\nabla}_{\phi}\nabla_{\phi}^{r-1} \overline{\nabla}_{\phi}^s {\rm Rm}_{\phi}|_{\omega_{\phi}}^2.
\end{equation}
We take a smooth cut-off function $\tau$ that is identically equal to $1$ on $\overline{B_{r''_{k+3}}(p)}$, vanishes on the outside of $B_{r'_{k+3}}(p)$ and satisfies
\[
|\partial \tau|_{\omega_0}, \;\; |\sqrt{-1}\partial \bar{\partial} \tau |_{\omega_0} \leq C_{62},
\]
where $r'_{k+3}>r''_{k+3}>r/2$. Applying the maximum principle to the function $\tau^2|\nabla_{\phi}^r \overline{\nabla}_{\phi}^s {\rm Rm}_{\phi}|_{\omega_{\phi}}^2+A_1|\nabla_{\phi}^{r-1} \overline{\nabla}_{\phi}^s {\rm Rm}_{\phi}|_{\omega_{\phi}}^2$ (for a suitable uniform constant $A_1$), we get
\[
|\nabla_{\phi}^r \overline{\nabla}_{\phi}^s {\rm Rm}_{\phi}|_{\omega_{\phi}}^2 \leq C_{63}
\]
on $\overline{B_{r''_{k+3}}(p)} \times [0,T]$.

\vspace{4mm}

\textbf{Case 2:} $s=0$.

Using the Ricci identity repeatedly, we have
\begin{eqnarray*}
\left( \frac{d}{dt}-\Delta_{\omega_{\phi}} \right) \nabla_{\phi}^{k+1}{\rm Rm}_{\phi}&=&\underbrace{\nabla_{\phi}^{k+1}\left( \frac{d}{dt}-\Delta_{\omega_{\phi}} \right) {\rm Rm}_{\phi}}_\text{($\nabla^{k+1}{\rm Rm}_{\phi}$;I)}+\underbrace{\sum_{\substack{p \geq 0, q \geq 1\\ p+q=k+1}} \nabla_{\phi}^p {\rm Rm}_{\phi} \ast \nabla_{\phi}^q ({\rm Rm}_{\phi}+\widetilde{\eta}+\nabla_{\phi}X)}_\text{($\nabla^{k+1}{\rm Rm}_{\phi}$;II)}\\
&+& \underbrace{\sum_{\substack{p \geq 0, q \geq 1\\ p+q=k+1}} \nabla_{\phi}^p{\rm Rm}_{\phi} \ast \nabla_{\phi}^q {\rm Rm}_{\phi}}_\text{($\nabla^{k+1}{\rm Rm}_{\phi}$;III)}.
\end{eqnarray*}
By \eqref{evc}, \eqref{soe}, \eqref{utc}, \eqref{sao} and the uniform bound of $|D^{k+1}U|_{\omega_{\phi}}^2$, we can estimate these terms as
\[
|\text{($\nabla^{k+1}{\rm Rm}_{\phi}$;I)}|_{\omega_{\phi}} \leq C_{64}(1+|\nabla_{\phi}^{k+1} {\rm Rm}_{\phi}|_{\omega_{\phi}}+|\nabla_{\phi}^{k+2} {\rm Rm}_{\phi}|_{\omega_{\phi}}),
\]
\[
|\text{($\nabla^{k+1}{\rm Rm}_{\phi}$;II)}|_{\omega_{\phi}}+|\text{($\nabla^{k+1}{\rm Rm}_{\phi}$;III)}|_{\omega_{\phi}} \leq C_{65}(1+|\nabla_{\phi}^{k+1} {\rm Rm}_{\phi}|_{\omega_{\phi}}).
\]
Thus we have
\begin{equation} \label{epc}
\left( \frac{d}{dt}-\Delta_{\omega_{\phi}} \right)|\nabla_{\phi}^{k+1}{\rm Rm}_{\phi}|_{\omega_{\phi}}^2 \leq C_{66}|\nabla_{\phi}^{k+1}{\rm Rm}_{\phi}|_{\omega_{\phi}}^2-\frac{1}{2}|\nabla_{\phi}^{k+2}{\rm Rm}_{\phi}|_{\omega_{\phi}}^2-|\overline{\nabla}_{\phi} \nabla_{\phi}^{k+1}{\rm Rm}_{\phi}|_{\omega_{\phi}}^2.
\end{equation}
Now we use the same cut-off function $\tau$ constructed in Case 1, and consider the function $\tau^2 |\nabla_{\phi}^{k+1}{\rm Rm}_{\phi}|_{\omega_{\phi}}^2+A_2 |\nabla_{\phi}^k{\rm Rm}_{\phi}|_{\omega_{\phi}}^2$ (for a suitable uniform constant $A_2$). Since the evolution equation of $|\nabla_{\phi}^k{\rm Rm}_{\phi}|_{\omega_{\phi}}^2$ has been already estimated in \eqref{emc}, the maximum principle implies that
\[
|\nabla_{\phi}^{k+1}{\rm Rm}_{\phi}|_{\omega_{\phi}}^2 \leq C_{67}
\]
on $\overline{B_{r''_{k+3}}(p)} \times [0,T]$. Combining with Case 1, we have
\[
|D^{k+1}{\rm Rm}_{\phi}|_{\omega_{\phi}}^2 \leq C_{68}
\]
on $\overline{B_{r''_{k+3}}(p)} \times [0,T]$, where the constant $C_{68}$ depends only on $N$, $\gamma$, $\omega_0$, $X$, $\|\phi(\cdot, 0)\|_{C^{k+5}(B_r (p))}$, $\|\phi\|_{C^0(B_r(p) \times [0,T])}$, $\|\widetilde{\eta}\|_{C^{k+3}(B_r(p))}$ and $\|F\|_{C^0(B_r(p))}$.

Applying $D^{k+1}$ to the equation \eqref{sMAc} and taking the trace, we have
\begin{eqnarray*}
|\Delta_{\omega_{\phi}} D^{k+1} \dot{\phi}|_{\omega_{\phi}} &\leq& |D^{k+1} \Delta_{\omega_{\phi}} \dot{\phi}|_{\omega_{\phi}}+C_{69} \sum_{i=0}^{k+1}|D^i {\rm Rm}_{\phi}|_{\omega_{\phi}} |D^{k+1-i} \dot{\phi}|_{\omega_{\phi}}\\
&\leq& C_{70} \left( |D^{k+1}{\rm Rm}_{\phi}|_{\omega_{\phi}}+|D^{k+1} \widetilde{\eta}|_{\omega_{\phi}}+|D^{k+2}X|_{\omega_{\phi}}+\sum_{i=0}^{k+1}|D^i {\rm Rm}_{\phi}|_{\omega_{\phi}} |D^{k+1-i} \dot{\phi}|_{\omega_{\phi}} \right).
\end{eqnarray*}
From the above estimates and $(H_k)$, we know that $|\Delta_{\omega_{\phi}} D^{k+1} \dot{\phi}|_{\omega_{\phi}}$ is uniformly bounded. Hence $D^{k+1} \dot{\phi}$ is $C^{1,\alpha}$, which implies $\dot{\phi}$ is $C^{k+2,\alpha}$. Differentiating the equation \eqref{MAf} $(k+2)$-times and applying the elliptic Schauder estimates, we find that $\phi$ is $C^{k+4, \alpha}$ on $\overline{B_{r_{k+4}}(p)} \times [0,T]$ where $r''_{k+3}>r_{k+4}>r/2$. Thus we have the statement $(H_{k+1})$ as desired. This completes the proof of Proposition \ref{toe}.
\end{proof}
Now we give the proof of Theorem \ref{tck}.
\begin{proof}[Proof of Theorem \ref{tck}]
Let $T>0$ be a constant. By Proposition \ref{ces}, we know that
\[
\sup_{M \times [0,T]}|\varphi_{\epsilon}|, \;\; \sup_{M \times [0,T]} |\dot{\varphi}_{\epsilon}|<C(T)
\]
for some constant $C(T)$ (independent of $\epsilon$). Thus Proposition \ref{ufl} implies that
\begin{equation} \label{ulr}
A(T)^{-1} \omega_{\epsilon} \leq \omega_{\varphi_{\epsilon}} \leq A(T) \omega_{\epsilon}
\end{equation}
on $M$ for some constant $A(T)$ (independent of $\epsilon$).
We exhaust $M\backslash D$ by a sequence of compact subsets $K$, and $[0,\infty)$ by a sequence of closed intervals $[0,T]$. From \eqref{ulr}, we know that
\[
N^{-1} \omega_0 \leq \omega_{\phi_{\epsilon}} \leq N \omega_0
\]
on $K \times [0,T]$, where the constant $N$ only depends on $K$ and $T$. Moreover, the initial data $k\chi+c_{\epsilon 0}$, $(1-\beta)\eta_{\epsilon}$, $F_{\epsilon}$ are uniformly bounded in the $C_{loc}^{\infty}$-topology on $K \times [0,T]$.
Thus Proposition \ref{toe}, together with the diagonal argument implies that there exists a subsequence $\varphi_{\epsilon_i}(t)$ which converges to a function $\varphi(t)$ that is smooth on $M \backslash D$. Then, by \eqref{ulr}, we also know that $\omega_{\varphi}$ is a conical K\"ahler metric along $(1-\beta)D$. Now we will check that $\omega_{\varphi}$ satisfies the equation \eqref{MCKRF}. Let $\zeta=\zeta(x,t)$ be any smooth $(n-1,n-1)$-form on $M \times [0,\infty)$ with compact support ${\rm Supp}(\zeta)$. Without loss of generality, we assume that ${\rm Supp}(\zeta) \subset [0,T)$. Since $F_{\epsilon}$, $\chi$, $\varphi_{\epsilon}$ are uniformly bounded on $M \times [0,T]$, for $t \in [0,T]$, dominated convergence theorem implies that
\begin{eqnarray*}
\int_M \frac{\partial \omega_{\varphi_{\epsilon}}}{\partial t} \wedge \zeta &=&
\int_M \sqrt{-1} \partial \bar{\partial} \left( \log \left( \frac{\omega_{\varphi_{\epsilon}}^n}{\omega_0^n} \cdot \prod_{i=1}^d(\epsilon^2+|s_i|_{H_i}^2)^{(1-\beta)\tau_i} \right)+F_0+\gamma(k\chi+\varphi_{\epsilon}) \right) \wedge \zeta\\
&\hbox{}&+\int_M L_X \omega_{\varphi_{\epsilon}} \wedge \zeta\\
&=& \int_M \left( \log \left( \frac{\omega_{\varphi_{\epsilon}}^n}{\omega_0^n} \cdot \prod_{i=1}^d(\epsilon^2+|s_i|_{H_i}^2)^{(1-\beta)\tau_i} \right)+F_0+\gamma(k\chi+\varphi_{\epsilon}) \right) \wedge \sqrt{-1} \partial \bar{\partial} \zeta\\
&\hbox{}&-\int_M \omega_{\varphi_{\epsilon}} \wedge L_X \zeta\\
&\xrightarrow{\epsilon_i \to 0}& \int_M \left( \log \frac{\omega_{\varphi}^n}{\omega_0^n}+F_0+\gamma(k\chi+\varphi) + \log |s_D|_{H_D}^{2(1-\beta)}\right) \wedge \sqrt{-1} \partial \bar{\partial} \zeta\\
&\hbox{}& -\int_M \omega_{\varphi} \wedge L_X \zeta\\
&=& \int_M \sqrt{-1} \partial \bar{\partial} \left( \log \frac{\omega_{\varphi}^n}{\omega_0^n}+F_0+\gamma(k\chi+\varphi) + \log |s_D|_{H_D}^{2(1-\beta)}\right) \wedge \zeta\\
&\hbox{}& +\int_M L_X \omega_{\varphi} \wedge \zeta\\
&=& \int_M (-{\rm Ric}(\omega_{\varphi})+\gamma \omega_{\varphi}+(1-\beta)[D]+L_X \omega_{\varphi}) \wedge \zeta, \\
\end{eqnarray*}
\[
\int_M \omega_{\varphi_{\epsilon_i}} \wedge \frac{\partial \zeta}{\partial t} \xrightarrow{\epsilon_i \to 0} \int_M \omega_{\varphi} \wedge \frac{\partial \zeta}{\partial t}.
\]
On the other hand, as in the proof of \cite[Theorem 4.1]{LZ17}, we have
\[
\int_M \frac{\partial \omega_{\varphi_{\epsilon}}}{\partial t} \wedge \zeta \xrightarrow{\epsilon_i \to 0} \int_M \frac{\partial \omega_{\varphi}}{\partial t} \wedge \zeta.
\]
Hence, on $[0,T]$, we find that
\begin{eqnarray*}
\frac{\partial}{\partial t} \int_M \omega_{\varphi} \wedge \zeta&=&\int_M(-{\rm Ric}(\omega_{\varphi})+\gamma \omega_{\varphi}+(1-\beta)[D]+L_X \omega_{\varphi}) \wedge \zeta \\
&\hbox{}& +\int_M \omega_{\varphi} \wedge \frac{\partial \zeta}{\partial t}.
\end{eqnarray*}
Integrating the above equation on $[0,\infty)$, we get
\begin{eqnarray*}
\int_{M \times [0,\infty)} \frac{\partial \omega_{\varphi}}{\partial t} \wedge \zeta dt &=& \int_0^{\infty} \left( \frac{\partial}{\partial t} \int_M \omega_{\varphi} \wedge \zeta - \int_M \omega_{\varphi} \wedge \frac{\partial \zeta}{\partial t} \right) dt \\
&=& \int_{M \times [0,\infty)} (-{\rm Ric}(\omega_{\varphi})+\gamma \omega_{\varphi}+(1-\beta)[D]+L_X \omega_{\varphi}) \wedge \zeta dt.
\end{eqnarray*}

Since $\zeta$ is arbitrary, $\omega_{\varphi}$ satisfies the equation \eqref{MCKRF} in the sense of distributions on $M \times [0,\infty)$. Meanwhile, the equation \eqref{MAF} can be written as
\[
\frac{(\omega_0+\sqrt{-1}\partial \bar{\partial} \phi_{\epsilon})^n}{\omega_0^n}=\frac{\exp(\dot{\phi_{\epsilon}}-F_0-\gamma \phi_{\epsilon}-\theta_X-X(\phi_{\epsilon}))}{\prod_{i=1}^d(\epsilon^2+|s_i|_{H_i}^2)^{(1-\beta)\tau_i}},
\]
where $\phi_{\epsilon}$, $\dot{\phi}_{\epsilon}$ and $X(\phi_{\epsilon})$ are uniformly bounded, which implies that the $L^p$-norm of the RHS is uniformly bounded for some $p>1$ since $\beta \in (0,1]$. Thus the H\"older continuity of $\varphi$ with respect to $\omega_0$ is a direct consequence from Kolodziej's work \cite[Theorem 2.1]{Kol08}. This completes the proof of Theorem \ref{tck}.
\end{proof}
%=========References===================================================

\end{document}